\documentclass[11pt,a4paper]{article}

\usepackage[left=1.8cm,right=1.8cm,top=1.8cm,bottom=1.8cm]{geometry}
\usepackage[fleqn]{amsmath}
\usepackage{amssymb,latexsym}
\usepackage{amsthm}
\usepackage[section]{placeins}
\usepackage[colorlinks=true,linkcolor=black,citecolor=black,urlcolor=black]{hyperref}
\usepackage{enumitem}
\usepackage{graphicx}
\usepackage{tikz,color}

\newtheorem{theorem}{Theorem}[section]
\newtheorem{corollary}[theorem]{Corollary}
\newtheorem{lemma}[theorem]{Lemma}
\newtheorem*{lemma*}{Lemma}

\theoremstyle{definition}
\newtheorem*{observation}{Observation}
\newtheorem*{example}{Example}
\newtheorem{problem}{Problem}

\DeclareMathOperator{\av}{av}
\newcommand{\eav}{\overline{\av}}
\newcommand{\edge}{\ }

\newcommand{\Z}{\mathbb{Z}}
\title{Mutually avoiding Eulerian circuits}

\author{
	Grahame Erskine\\ \texttt{\small grahame.erskine@open.ac.uk}
	\and Terry Griggs\\ \texttt{\small terry.griggs@open.ac.uk}
	\and Robert Lewis\\ \texttt{\small robert.lewis@open.ac.uk}
	\and James Tuite\\ \texttt{\small james.t.tuite@open.ac.uk}\\[1ex]
    School of Mathematics and Statistics, The Open University, Milton Keynes, UK
}

\date{}

\begin{document}
	\maketitle
	\let\thefootnote\relax\footnote{Mathematics subject classification: 05C45}
	\let\thefootnote\relax\footnote{Keywords: Eulerian circuit; avoidance}
	
	\vspace*{-4ex}
	\begin{abstract}\noindent
		Two Eulerian circuits, both starting and ending at the same vertex, are \emph{avoiding} if at every other point of the circuits they are at least distance 2 apart. An Eulerian graph which admits two such avoiding circuits starting from any vertex is said to be \emph{doubly Eulerian}. The motivation for this definition is that the extremal Eulerian graphs, i.e. the complete graphs on an odd number of vertices and the cycles, are not doubly Eulerian. We prove results about doubly Eulerian graphs and identify those that are the ``densest'' and ``sparsest'' in terms of the number of edges.
	\end{abstract}
	
	\section{Introduction}\label{sec:intro}
	The problem of traversing a network, subject to certain constraints, in an efficient manner is a fundamental concept in graph theory and related disciplines. Examples such as the \emph{travelling salesman} problem are well known, where the constraint is to visit each node of the network at least once (see for example~\cite[Ch.15]{J}). A related problem is known as the \emph{Chinese postman} or \emph{route optimisation} problem~\cite{K} (also see~\cite[Ch.14]{J}), where the constraint is to traverse each edge in the network at least once. 
	
	The route optimisation problem is clearly trivial if the graph to be traversed is Eulerian. In that case, we propose here an extension of the problem to the case where the network is to be traversed by \emph{two} postmen simultaneously, both starting and finishing at the same node on the network, and both traversing each edge exactly once. The constraint we impose is that the two postmen should not be aware of each other's position on the route, except at the start and end points. By this we mean that except at the start and end points, at any given point in time the two postmen should neither be at the same vertex nor at adjacent vertices. Of course, we assume that each postman will traverse exactly one edge in one unit of time. One might ask why we do not allow the postmen to be at adjacent vertices: one motivation is as follows. Consider the simple case of a cycle graph of even order. Clearly the two postmen would have to set off in opposite directions round the cycle, and would therefore meet at the opposite vertex. But in the case of an odd cycle, the two would never be at the same vertex, but would `meet' in the sense that they would traverse the same edge in opposing directions. Our insistence that the postmen are not allowed to be at adjacent vertices means that the situation for odd and even cycles is consistent. Informally, the two postmen can neither meet at the same vertex nor be able to see one another at adjacent vertices by looking along an edge.
	
	We formulate the problem using standard terminology from graph theory; for definitions and notation not noted here see a standard text, for example~\cite{BM}. Throughout, let $G=(V,E)$ be a connected graph without loops and multiple edges where $V$ is the set of vertices and $E$ is the set of edges. A \emph{trail} in a graph is a sequence of vertices $v_1,v_2,\ldots v_m$ such that $v_i$ is adjacent to $v_{i+1}$ for $1\leq i\leq m-1$ and no edge of the graph is traversed more than once. A \emph{circuit} in a graph is a closed trail, i.e. a trail beginning and ending at the same vertex. An \emph{Eulerian circuit} in a graph is a circuit in which each edge is traversed exactly once. We say a graph is \emph{Eulerian} if it admits an Eulerian circuit; it is of course a foundational result in graph theory~\cite{E} that a graph is Eulerian if and only if every vertex has even valency.
	
	Let $G$ be an Eulerian graph with $m$ edges and let $u$ be a vertex of $G$. Let $C_1=u,v_1,v_2,\ldots v_{m-1},u$ and $C_2=u,w_1,w_2,\ldots w_{m-1},u$ be two Eulerian circuits in $G$. We say $C_1$ and $C_2$ are \emph{avoiding} if two criteria are met: (1) $v_i \neq w_i$, $1 \leq i \leq m-1$ and  (2) $v_i$ is not adjacent to $w_i$, $1\leq i\leq m-1$. Thus two circuits are avoiding if they are at distance at least 2 at every step apart from the beginning and end points. We say $G$ is \emph{doubly Eulerian} if it admits a pair of avoiding Eulerian circuits starting from \emph{any} vertex.
	
	The motivation for this definition is that the extremal Eulerian graphs, i.e. the complete graphs $K_{2n+1}$, $n \geq 1$, which have the largest ratio of edges to vertices of $n$, and the cycles $C_n$, $n \geq 3$, which have the smallest such ratio of $1$, are not doubly Eulerian.  This leads naturally to the question of identifying what are the ``densest'' and the ``sparsest'' doubly Eulerian graphs, and much of our paper is concerned with these questions.
	
	We can extend the concept of a doubly Eulerian graph further. Given an Eulerian graph $G$, we define the \emph{avoidance index} $\av(G)$ to be the largest integer $k$ such that starting from any vertex of $G$ we may find a set of $k$ mutually avoiding Eulerian circuits. A graph $G$ is therefore doubly Eulerian if and only if $\av(G)\geq 2$. 
	
    For small orders, we can calculate the avoidance index of Eulerian graphs directly by computer. Table~\ref{tab:small} shows the results. We see that no graph of order less than 8 is doubly Eulerian. Of the doubly Eulerian graphs of order 8, five of the seven are regular of valency 4, including the single example of avoidance index 3; we discuss these in Section~\ref{sec:maximal}. The remaining two graphs are discussed in Section~\ref{sec:minimal}. No graph of order 9 has avoidance index greater than 2. At order 10, the single example at avoidance index 4 is $K^*_{5,5}$ which is the complete bipartite graph $K_{5,5}$ minus a perfect matching; this graph is discussed in Section~\ref{sec:index}. The two order 10 graphs of avoidance index 3 are $K_{4,6}$ and the circulant graph with generating set $\{1,4\}$. The computations required to calculate the avoidance index are substantial, especially when there are a number of vertices of valency 6. We have therefore not been able to determine the split between avoidance index 1 and 2 for graphs of order 10.
	
	\begin{table}\centering
		\begin{tabular}{|c|c|cccc|}
			\hline
			Order & Eulerian & \multicolumn{4}{|c|}{Avoidance index} \\
			& graphs & 1 & 2 & 3 & 4 \\
			\hline
			3 & 1 & 1 & 0 & 0 & 0 \\
			4 & 1 & 1 & 0 & 0 & 0 \\
			5 & 4 & 4 & 0 & 0 & 0 \\
			6 & 8 & 8 & 0 & 0 & 0 \\
			7 & 37 & 37 & 0 & 0 & 0 \\
			8 & 184 & 177 & 6 & 1 & 0 \\
			9 & 1782 & 1692 & 90 & 0 & 0 \\
			10 & 31026 & ? & ? & 2 & 1 \\
			\hline
		\end{tabular}
		\caption{Avoidance index of small graphs}
		\label{tab:small}
	\end{table}
	
	The remainder of this paper is organised as follows. Section~\ref{sec:maximal} is concerned with doubly Eulerian graphs having the greatest number of possible edges for a given order, which we will call \emph{edge-maximal}; we give a bound on the possible number of edges in a doubly Eulerian graph and show that this bound can be attained for all orders 8 and above. 
	
	Bipartite graphs are of particular interest because the second criterion that $v_i$ is not adjacent to $w_i$, $1 \leq i \leq m-1$, automatically follows from the first criterion that  $v_i \neq w_i$, $1 \leq i \leq m-1$, and so becomes redundant. This is the subject matter of Section~\ref{sec:bipartite} where we give bounds for this special case, and show that again these bounds are attained.
	
	Section~\ref{sec:minimal} deals with doubly Eulerian graphs having the least number of possible edges which we will call \emph{edge-minimal}. The analysis of possible edge-minimal graphs is significantly more complicated than the edge-maximal case, and we restrict ourselves to constructions for certain special cases.
	
	We study the avoidance index in Section~\ref{sec:index}, and finally in Section~\ref{sec:summary}, we summarise our results and suggest a number of open questions and possible directions for future research.
	
	\section{Edge-maximal doubly Eulerian graphs}\label{sec:maximal}
	It seems intuitively clear that in some sense, very dense graphs are unlikely to be doubly Eulerian because it is difficult for the two postmen to avoid being adjacent at some point on the circuit. More formally, we may ask for the maximum possible number of edges in a doubly Eulerian graph of given order. We begin with two simple lemmas.
	\begin{lemma}\label{oddnvaln-1}
		There exists no doubly Eulerian graph of odd order $n\geq 3$ containing a vertex of valency $n-1$.
	\end{lemma}
	\begin{proof}
		Let $G$ be an Eulerian graph of odd order $n\ge 3$ with vertex $v$ with valency $n-1$. Then $v$ is adjacent to every other vertex of $G$. Hence any two Eulerian circuits, beginning and ending at a vertex other than $v$, fail to be avoiding at $v$. 
	\end{proof}
	\begin{lemma}\label{evennvaln-2}
		There exists no doubly Eulerian graph of even order $n\ge 4$ containing a vertex of valency $n-2$.
	\end{lemma}
	\begin{proof}
		Let $G$ be an Eulerian graph of even order $n\ge 4$ with vertex $u$ with valency $n-2$. Then there is exactly one vertex, $v$, not adjacent to $u$. Consider two Eulerian avoiding circuits $C_1$ and $C_2$ starting at a common vertex. At every step, if $C_1$ visits $u$ then $C_2$ must visit $v$. Hence, there are at least as many edges incident to $v$ as to $u$, and so $v$ also has valency $n-2$. Now consider two Eulerian circuits starting at vertex $u$. Before both return to $u$ for their final visit to complete the circuits, they will need to visit $v$ once more than $u$. But this is impossible as the visits to $u$ and $v$ occur as complementary pairs.
	\end{proof}
	Given the restrictions imposed by the above lemmas, the maximum possible number of edges in a doubly Eulerian graph of even order $n$ would be attained by an $(n-4)$-regular graph; and in such a graph of odd order $n$ by an $(n-3)$-regular graph. From Table~\ref{tab:small} we know that the smallest doubly Eulerian graph has order 8, and indeed there are five examples of 4-regular graphs of that order, see Figure~\ref{fig:doubly_8}. Graph~(e) is the complete bipartite graph $K_{4,4}$ and actually has avoidance index 3.
	
	\begin{figure}
		\centering
		\begin{tabular}{ccc}
			\begin{tikzpicture}[x=0.2mm,y=-0.2mm,inner sep=0.2mm,scale=0.8,thick,vertex/.style={circle,draw,minimum size=10,fill=lightgray}]
\node at (279,318) [vertex] (v1) {};
\node at (387,426) [vertex] (v2) {};
\node at (387,273) [vertex] (v3) {};
\node at (279,381) [vertex] (v4) {};
\node at (432,318) [vertex] (v5) {};
\node at (324,273) [vertex] (v6) {};
\node at (324,426) [vertex] (v7) {};
\node at (432,381) [vertex] (v8) {};
\path
	(v1) edge (v4)
	(v1) edge (v5)
	(v1) edge (v6)
	(v1) edge (v7)
	(v2) edge (v5)
	(v2) edge (v6)
	(v2) edge (v7)
	(v2) edge (v8)
	(v3) edge (v5)
	(v3) edge (v6)
	(v3) edge (v7)
	(v3) edge (v8)
	(v4) edge (v6)
	(v4) edge (v7)
	(v4) edge (v8)
	(v5) edge (v8)
	;
\end{tikzpicture} & \begin{tikzpicture}[x=0.2mm,y=-0.2mm,inner sep=0.2mm,scale=0.9,thick,vertex/.style={circle,draw,minimum size=10,fill=lightgray}]
\node at (443,339) [vertex] (v1) {};
\node at (341,440) [vertex] (v2) {};
\node at (341,297) [vertex] (v3) {};
\node at (443,398) [vertex] (v4) {};
\node at (401,297) [vertex] (v5) {};
\node at (300,398) [vertex] (v6) {};
\node at (401,440) [vertex] (v7) {};
\node at (300,339) [vertex] (v8) {};
\path
	(v1) edge (v4)
	(v1) edge (v5)
	(v1) edge (v6)
	(v1) edge (v7)
	(v2) edge (v4)
	(v2) edge (v6)
	(v2) edge (v7)
	(v2) edge (v8)
	(v3) edge (v5)
	(v3) edge (v6)
	(v3) edge (v7)
	(v3) edge (v8)
	(v4) edge (v5)
	(v4) edge (v7)
	(v5) edge (v8)
	(v6) edge (v8)
	;
\end{tikzpicture} & \begin{tikzpicture}[x=0.2mm,y=-0.2mm,inner sep=0.2mm,scale=0.75,thick,vertex/.style={circle,draw,minimum size=10,fill=lightgray}]
\node at (234,290) [vertex] (v1) {};
\node at (287,416) [vertex] (v2) {};
\node at (361,237) [vertex] (v3) {};
\node at (361,416) [vertex] (v4) {};
\node at (413,290) [vertex] (v5) {};
\node at (234,364) [vertex] (v6) {};
\node at (413,364) [vertex] (v7) {};
\node at (287,237) [vertex] (v8) {};
\path
	(v1) edge (v4)
	(v1) edge (v5)
	(v1) edge (v6)
	(v1) edge (v8)
	(v2) edge (v4)
	(v2) edge (v6)
	(v2) edge (v7)
	(v2) edge (v8)
	(v3) edge (v5)
	(v3) edge (v6)
	(v3) edge (v7)
	(v3) edge (v8)
	(v4) edge (v5)
	(v4) edge (v7)
	(v5) edge (v7)
	(v6) edge (v8)
	;
\end{tikzpicture} \\ 
			(a) & (b) & (c) \\
			\begin{tikzpicture}[x=0.2mm,y=-0.2mm,inner sep=0.2mm,scale=0.8,thick,vertex/.style={circle,draw,minimum size=10,fill=lightgray}]
\node at (326,454) [vertex] (v1) {};
\node at (439,340) [vertex] (v2) {};
\node at (279,406) [vertex] (v3) {};
\node at (392,294) [vertex] (v4) {};
\node at (392,454) [vertex] (v5) {};
\node at (279,340) [vertex] (v6) {};
\node at (439,407) [vertex] (v7) {};
\node at (326,293) [vertex] (v8) {};
\path
	(v1) edge (v3)
	(v1) edge (v5)
	(v1) edge (v6)
	(v1) edge (v7)
	(v2) edge (v4)
	(v2) edge (v5)
	(v2) edge (v7)
	(v2) edge (v8)
	(v3) edge (v5)
	(v3) edge (v6)
	(v3) edge (v8)
	(v4) edge (v6)
	(v4) edge (v7)
	(v4) edge (v8)
	(v5) edge (v7)
	(v6) edge (v8)
	;
\end{tikzpicture} & \begin{tikzpicture}[x=0.2mm,y=-0.2mm,inner sep=0.2mm,scale=0.8,thick,vertex/.style={circle,draw,minimum size=10,fill=lightgray}]
\node at (326,454) [vertex] (v1) {};
\node at (439,340) [vertex] (v2) {};
\node at (279,406) [vertex] (v3) {};
\node at (392,294) [vertex] (v4) {};
\node at (392,454) [vertex] (v5) {};
\node at (279,340) [vertex] (v6) {};
\node at (439,407) [vertex] (v7) {};
\node at (326,293) [vertex] (v8) {};
\path
	(v1) edge (v3)
	(v1) edge (v5)
	(v1) edge (v2)
	(v1) edge (v8)
	(v2) edge (v4)
	(v2) edge (v6)
	(v2) edge (v7)
	(v3) edge (v4)
	(v3) edge (v6)
	(v3) edge (v7)
	(v4) edge (v5)
	(v4) edge (v8)
	(v5) edge (v6)
	(v5) edge (v7)
	(v6) edge (v8)
	(v7) edge (v8)
	;
\end{tikzpicture} &  \\ 
			(d) & (e) &  \\
		\end{tabular}
		
		\caption{The five 4-regular doubly Eulerian graphs of order 8}
		\label{fig:doubly_8}
	\end{figure}
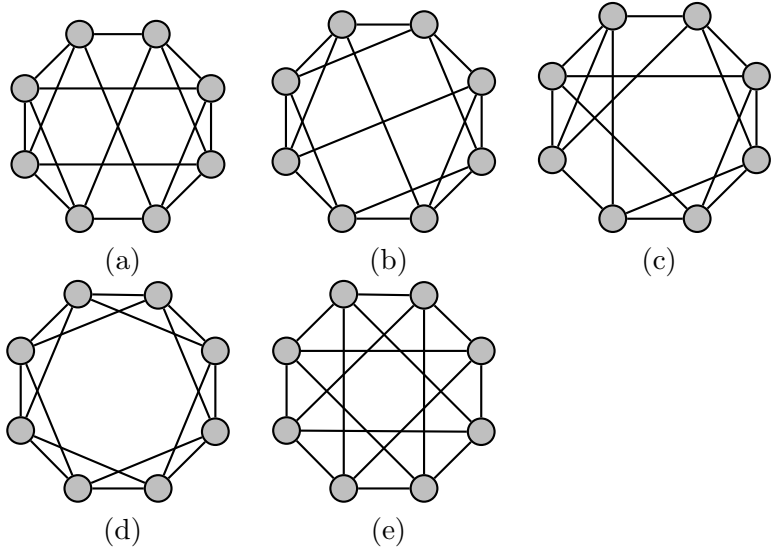
	
	It is interesting to note that there are exactly six 4-regular graphs of order 8, so only one of these fails to be doubly Eulerian. The exceptional graph turns out to be the complement of the cube, and is shown in Figure~\ref{fig:notdoubly_8}.
	
	\begin{figure}
		\centering
		\begin{tikzpicture}[x=0.2mm,y=0.2mm,thick,vertex/.style={circle,draw,minimum size=10,inner sep=0,fill=lightgray}]
	\node at (67.1,29.1) [vertex] (v1) {};
	\node at (27.7,-65.8) [vertex] (v2) {};
	\node at (-27.9,-65.8) [vertex] (v3) {};
	\node at (-67.1,29.1) [vertex] (v4) {};
	\node at (-67.2,-26.4) [vertex] (v5) {};
	\node at (-27.8,68.4) [vertex] (v6) {};
	\node at (27.8,68.4) [vertex] (v7) {};
	\node at (67.1,-26.6) [vertex] (v8) {};
	\draw (v1) to (v4);
	\draw (v1) to (v6);
	\draw (v1) to (v7);
	\draw (v1) to (v8);
	\draw (v2) to (v3);
	\draw (v2) to (v5);
	\draw (v2) to (v7);
	\draw (v2) to (v8);
	\draw (v3) to (v5);
	\draw (v3) to (v6);
	\draw (v3) to (v8);
	\draw (v4) to (v5);
	\draw (v4) to (v6);
	\draw (v4) to (v7);
	\draw (v5) to (v8);
	\draw (v6) to (v7);
\end{tikzpicture}
		
		\caption{The unique 4-regular graph of order 8 which is not doubly Eulerian}
		\label{fig:notdoubly_8}
	\end{figure}
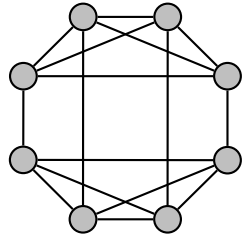
	
	Our next results show that these edge-maximal graphs do in fact exist for all orders greater than or equal to 8. The graphs that we prove to be edge-maximal in the following theorems are circulant graphs, and therefore vertex-transitive. Thus it suffices to find avoiding Eulerian circuits starting and ending from any vertex. The case where $n=6s+3,s\geq 1$, is relatively easy.
	
	\begin{theorem}\label{thm:3mod6}
		For any $n\equiv 3\pmod 6$ with $n\geq 9$, there exists an edge-maximal doubly Eulerian graph of order $n$.
	\end{theorem}
	\begin{proof}
		By Lemma~\ref{oddnvaln-1}, such a graph must be regular of degree $n-3$. Let $G$ be the complete graph $K_n$ on vertex set $\{0,1,\ldots,n-1\}$. Delete from $G$ all edges $i\edge(i+n/3)$, $0\leq i\leq n-1$, with arithmetic mod $n$. The reduced graph $G'$ is regular of degree $n-3$. Choose vertices $x,y\notin\{0,n/3,2n/3\}$, with $x\neq y$, and remove edges $0\edge x$, $0\edge y$, $(n/3)\edge x$ and $(n/3)\edge y$. The graph remains Eulerian, so choose an Eulerian circuit and separate it into two Eulerian trails $T_1$ and $T_2$, both starting at vertex $x$ and ending at vertex $y$. Construct an Eulerian circuit of $G'$ as follows:
		
		$\quad 0,x,T_1,y,(n/3),x,T_2,y,0$.

  		Further construct a second ``parallel'' Eulerian circuit:
		
		$\quad (2n/3),(x+2n/3),T_1^{(+2n/3)},(y+2n/3),0,(x+2n/3),T_2^{(+2n/3)},(y+2n/3),(2n/3)$,
		
		where if the first circuit visits vertex $z$ at a particular step, the second visits vertex $(z+2n/3)$, arithmetic modulo $n$. Now replace $(2n/3)$ with 0 at the start and end points of the second circuit, and replace the path $(y+2n/3),0,(x+2n/3)$ with the path $(y+2n/3),(2n/3),(x+2n/3)$. The result is a pair of avoiding Eulerian circuits beginning and ending at 0. 
  	\end{proof}
	The remaining odd orders require a more complex construction.
	\begin{theorem}\label{thm:oddncirc}
		For any odd $n\geq 9$, there exists an edge-maximal doubly Eulerian graph of order $n$ which is regular of degree $n-3$.
	\end{theorem}
	\begin{proof}
		Let $X(V,G)$ be a circulant graph of odd order $n\ge9$ and even degree $d=n-3$ defined as a Cayley graph with vertex set $V=\Z_n$ and generator set $G=\{b_1,\dots, b_f\}$  where $f=d/2$ and each $b_i\le (n-1)/2$, so that the connection set is $\{\pm b_1, \dots, \pm b_f\}$. As $X$ is vertex-transitive, it suffices to demonstrate the existence of two avoiding Eulerian circuits $C_1$ and $C_2$ from an arbitrary vertex $u$. 
		
		For odd order $n$, where $ 9\leq n \leq 21$, we have the following constructions with circuits starting at vertex 0. It is easily verified that these circuits are Eulerian, also that the difference between each pair of vertices is not equal to a connection set element and not zero except at the endpoints. 
		
		Order $n=9, G=\{1,2,3\}$:
		
		Circuit $C_1$: 0	1	2	3	4	5	6	7	8	0	7	5	3	6	8	2	5	8	1	7	4	6	0	2	4	1	3	0
		
		Circuit $C_2$: 0	6	7	8	0	1	2	3	4	5	3	1	8	2	4	7	1	4	6	3	0	2	5	6	8	5	7	0.
		
		Order $n=11, G=\{1,2,3,4\}$:
		
		Circuit $C_1$: 0	1	2	3	4	5	6	7	8	9	10	0	2	4	6	8	10	1	3	5	7	9	0	4	1	5	2	10	7	4	8	0	7	3	6	9	1	8	5	9	2	6	10	3	0

		Circuit $C_2$: 0	7	8	9	10	0	1	2	3	4	5	6	7	9	0	2	4	6	8	10	1	3	5	9	6	10	7	5	2	10	3	6	2	9	1	4	7	3	0	4	8	1	5	8	0.
		
\newpage	
  Order $n=13, G=\{2,3,4,5,6\}$:
		
		Circuit $C_1$: 0	4	6	8	10	12	1	3	5	7	9	11	0	2	4	7	10	0	3	6	9	12	2	5	8	11	1	4	8	12	3	7	11	2	6	10	1	5	9	0	5	10	3	8	1	6	11	3	9	1	7	12	5	11	4	9	2	7	0	6	12	4	10	2	8	0
		
		Circuit $C_2$: 0	5	7	9	11	0	2	4	6	8	10	12	1	3	5	8	11	1	4	7	10	0	3	6	9	12	2	5	9	0	4	8	12	3	7	11	2	6	10	1	6	11	4	9	2	7	12	4	10	2	8	0	6	12	5	10	3	8	1	5	11	3	9	1	7	0.
		
		Order $n=15, G=\{1,2,3,4,5,6\}$:
		
		Circuit $C_1$: 0	4	5	14	4	1	14	3	4	13	3	0	13	2	3	12	2	14	12	1	2	11	1	13	11	0	1	10	0	12	10	14	0	9	14	11	9	13	14	8	13	10	8	12	13	7	12	9	7	11	12	6	11	8	6	10	11	5	10	7	5	9	10	4	9	6	4	8	9	3	8	5	3	7	8	2	7	4	2	6	7	1	6	3	1	5	6	0	5	2	0

		Circuit $C_2$: 0	12	13	7	12	9	7	11	12	6	11	8	6	10	11	5	10	7	5	9	10	4	9	6	4	8	9	3	8	5	3	7	8	2	7	4	2	6	7	1	6	3	1	5	6	0	5	2	0	4	5	14	4	1	14	3	4	13	3	0	13	2	3	12	2	14	12	1	2	11	1	13	11	0	1	10	14	11	9	14	0	9	13	10	8	13	14	8	12	10	0.
		
		Order $n=17, G=\{1,2,3,4,5,6,7\}$:
		
		Circuit $C_1$: 0	11	12	8	3	13	15	1	12	13	9	4	14	16	2	13	14	10	5	15	0	3	14	15	11	6	16	1	4	15	16	12	7	0	2	5	16	0	13	8	1	3	6	0	1	14	9	2	4	7	1	2	15	10	3	5	8	2	3	16	11	4	6	9	3	4	0	12	5	7	10	4	5	1	13	6	8	11	5	6	2	14	7	9	12	6	7	3	15	8	10	13	7	8	4	16	9	11	14	8	9	5	0	10	12	15	9	10	6	1	11	13	16	10	11	7	2	12	14	0
		
		Circuit $C_2$: 0	2	3	16	11	4	6	9	3	4	0	12	5	7	10	4	5	1	13	6	8	11	5	6	2	14	7	9	12	6	7	3	15	8	10	13	7	8	4	16	9	11	14	8	9	5	0	10	12	15	9	10	6	1	11	13	16	10	11	7	2	12	14	0	11	12	8	3	13	15	1	12	13	9	4	14	16	2	13	14	10	5	15	0	3	14	15	11	6	16	1	4	15	16	12	7	1	3	5	16	0	13	8	2	4	7	0	1	14	9	2	5	8	1	2	15	10	3	6	0.

    	Order $n=19, G=\{1,2,3,4,5,6,7,8\}$:

        Circuit $C_1$: 0 12 11 15 10 4 15 17 1 13 12 16 11 5 16 18 2 14 13 17 12 6 17 0 3 15 14 18 13 7 18 1 4 16 15 0 14 8 0 2 5 17 16 1 15 9 1 3 6 18 17 2 16 10 2 4 7 0 18 3 17 11 3 5 8 1 0 4 18 12 4 6 9 2 1 5 0 13 5 7 10 3 2 6 1 14 6 8 11 4 3 7 2 15 7 9 12 5 4 8 3 16 8 10 13 6 5 9 4 17 9 11 14 7 6 10 5 18 10 12 15 8 7 11 6 0 11 13 16 9 8 12 7 1 12 14 17 10 9 13 8 2 13 15 18 11 10 14 9 3 14 16 0
        
        Circuit $C_2$: 0 2 1 5 0 13 5 7 10 3 2 6 1 14 6 8 11 4 3 7 2 15 7 9 12 5 4 8 3 16 8 10 13 6 5 9 4 17 9 11 14 7 6 10 5 18 10 12 15 8 7 11 6 0 11 13 16 9 8 12 7 1 12 14 17 10 9 13 8 2 13 15 18 11 10 14 9 3 14 16 0 12 11 15 10 4 15 17 1 13 12 16 11 5 16 18 2 14 13 17 12 6 17 0 3 15 14 18 13 7 18 1 4 16 15 0 14 8 1 3 5 17 16 1 15 9 2 4 6 18 17 2 16 10 2 5 8 0 18 3 17 11 3 6 9 1 0 4 18 12 4 7 0
        
    	Order $n=21, G=\{1,2,3,4,5,6,7,8,9\}$:
     
        Circuit $C_1$: 0 13 12 8 13 19 5 17 19 1 14 13 9 14 20 6 18 20 2 15 14 10 15 0 7 19 0 3 16 15 11 16 1 8 20 1 4 17 16 12 17 2 9 0 2 5 18 17 13 18 3 10 1 3 6 19 18 14 19 4 11 2 4 7 20 19 15 20 5 12 3 5 8 0 20 16 0 6 13 4 6 9 1 0 17 1 7 14 5 7 10 2 1 18 2 8 15 6 8 11 3 2 19 3 9 16 7 9 12 4 3 20 4 10 17 8 10 13 5 4 0 5 11 18 9 11 14 6 5 1 6 12 19 10 12 15 7 6 2 7 13 20 11 13 16 8 7 3 8 14 0 12 14 17 9 8 4 9 15 1 13 15 18 10 9 5 10 16 2 14 16 19 11 10 6 11 17 3 15 17 20 12 11 7 12 18 4 16 18 0
        
        Circuit $C_2$: 0 2 1 18 2 8 15 6 8 11 3 2 19 3 9 16 7 9 12 4 3 20 4 10 17 8 10 13 5 4 0 5 11 18 9 11 14 6 5 1 6 12 19 10 12 15 7 6 2 7 13 20 11 13 16 8 7 3 8 14 0 12 14 17 9 8 4 9 15 1 13 15 18 10 9 5 10 16 2 14 16 19 11 10 6 11 17 3 15 17 20 12 11 7 12 18 4 16 18 0 13 12 8 13 19 5 17 19 1 14 13 9 14 20 6 18 20 2 15 14 10 15 0 7 19 0 3 16 15 11 16 1 8 20 1 4 17 16 12 17 2 9 1 3 5 18 17 13 18 3 10 2 4 6 19 18 14 19 4 11 2 5 7 20 19 15 20 5 12 3 6 9 0 20 16 0 6 13 4 7 10 1 0 17 1 7 14 5 8 0
		
		For odd order $n\geq 23$ and even degree $d=n-3$, the circulant graph $X(V,G)$ is defined as a Cayley graph with vertex set $V=\Z_n$ and generator set $G=\{1, 2, \dots, f\}$ where $f=d/2$, so that the set of connection elements is $\{\pm 1, \dots, \pm f\}$. Each edge of the graph is defined by a vertex and a connection element, so that each circuit may be defined by an initial vertex and a sequence of connection elements.
		
		Given any vertex of $X$, the only vertices not adjacent to it are the ones with difference $f+1$ and $f+2$ as these are not in the generator set. Therefore the two circuits must be constructed such that the difference between the pair of vertices at each step is one of these two values. The number of edges in each circuit is $nf$, and the circuits are constructed by concatenating $n$ paths of length $f$. Define a sequence $C=(c_1,c_2,\ldots,c_f)$ where the $c_i$ are a permutation of the generators $1,2,\ldots,f$ with the restriction that $c_1=f-1$, $c_{f-2}=f$, $c_{f-1}=2$ and $c_f=3$. Such a sequence can be taken to be, for example, $(f-1,1,4,5,\ldots,f-2,f,2,3)$. Then from an arbitrary vertex $v$, define a path $P$ of length $f$ by a sequence $B=(b_1,b_2,\ldots,b_f)$ of connection elements $b_i=\pm c_i$ where $b_1=-c_1$, $b_{f-2}=-c_{f-2}$, $b_{f-1}=c_{f-1}$, $b_f=c_f$ and for $i=2$ to $f-3$, $b_i=c_i$ for the values of $c_i$ specified in Table~\ref{table:330} and $b_i=-c_i$ otherwise. This ensures that $\sum_{i=1}^f b_i\equiv 1\pmod{n}$.
		
		\begin{table}[h]
	
			\setlength{\tabcolsep}{10pt}
			\centering 
			\begin{tabular} {|c l l l |}
				\hline 
				Order $n$ & \multicolumn{3}{l|}{Values of $c_i$ in the set $\{1, 4, \dots, f-2\}$ }   \\ 
				$\pmod 8$ & $c_i\le9$ & $c_i\ge10$ & Isolated value \\
				\hline 
				1 & 4, 5, 7 & $1, 2 \pmod4$ & - \\
				3 & 4, 7 & $2, 3 \pmod4$ & - \\ 
				5 & 5, 6, 7, 9 & $1, 2 \pmod4$ & $(n-21)/2$ \\ 
				7 & 1, 5, 6, 7 & $2, 3 \pmod 4$ & $(n-15)/2$ \\ 
				\hline 
			\end{tabular}
		      \caption{For $i=2$ to $f-3$, values of $c_i$ for which $b_i=c_i$}
                \label{table:330} 
		\end{table}
		
		Connection block $B$ includes one instance of each generator, taken either positive or negative. Starting at an arbitrary vertex $v\in\Z_n$, it defines a path $P$ ending at vertex $v+1$ (all arithmetic modulo $n$). We construct a path $C_1$ from vertex $u$ by the concatenation of $n$ paths $P_i$, $i=1,\dots, n$, defined by instances of $B$, labelled $B_i$. Then $C_1$ is a circuit of length $nf$. Consider an arbitrary edge $e$ of $X$, between vertices $v_1$ and $v_2$. Let $c=v_2-v_1$, so that $c$ is an element of the connection set and there is a generator $b_j\in G$ such that $b_j=\pm c$. Let $h=\sum_{i=1}^{j-1} b_i$. If $b_j=c$, then connection element $c$ defines an edge of $B_1$ from vertex $u+h$. Let $k=v_1+1-(u+h)$. Then, as $\sum_{i=1}^f b_i \equiv 1 \pmod{n}$, we have $e\in B_k$. Similarly, if $b_j=-c$, then $e\in B_k$ where $k=v_2+1-(u+h)$. So $X\subset C_1$ and therefore $C_1$ is an Eulerian circuit of $X$.
		
		Circuit $C_2$ starts from vertex $u$ with a path defined by connection block $B$, except that the first connection element is 2 instead of $-(f-1)$ in order to establish a difference of $f+1$; this is maintained by concatenating a further $f+5$ paths defined by block $B$. The rest of $C_2$ is constructed from $f-3$ paths defined by variants of block $B$. The $n$ concatenated paths of both circuits are defined in Table \ref{table:331}.
		
		\begin{table}[h]
			\setlength{\tabcolsep}{7pt}
			\centering 
			\begin{tabular} {|c c l l l l l|}
				\hline 
				Path & Number & \multicolumn{5}{l|}{Connection elements 1 to $f$} \\ 
				sequence & of paths & 1 & 2 to $f-3$ & $f-2$ & $f-1$ & $f$ \\
				\hline 
				\multicolumn{7}{|l|}{Circuit $C_1$} \\
				All (1 to $n$) & $n$ & $-(f-1)$ & $b_2$ to $b_{f-3}$ & $-f$ & 2 & 3 \\[1ex]
				\multicolumn{7}{|l|}{Circuit $C_2$} \\
				1 & 1 & 2 & $b_2$ to $b_{f-3}$ & $-f$ & 2 & 3 \\
				2 to $f+6$ & $f+5$ & $-(f-1)$ & $b_2$ to $b_{f-3}$ & $-f$ & 2 & 3 \\ 
				$f+7$ and $f+8$  & 2 & $-(f-1)$ & $b_2$ to $b_{f-3}$ & $-(f-1)$ & 2 & 2 \\ 
				$f+9$ to $2f$ & $f-8$ & $-(f-1)$ & $b_2$ to $b_{f-3}$ & $-f$ & 3 & 2 \\ 
				$2f+1$ & 1 & $-(f-1)$ & $b_2$ to $b_{f-3}$ & $-f$ & 3 & 3 \\ 
				$2f+2$ & 1 & $-f$ & $b_2$ to $b_{f-3}$ & $-f$ & 3 & 3 \\ 
				$2f+3$ & 1 & $-f$ & $b_2$ to $b_{f-3}$ & $-f$ & 3 & $-(f-1)$ \\ 
				\hline 
			\end{tabular}
			\caption{The connection elements for circuits $C_1$ and $C_2$ ($n$ paths of $f$ elements)}
			\label{table:331} 
		\end{table}
		
		With these paths, the difference between the pair of vertices at each step of circuits $C_1$ and $C_2$ is maintained at $f+1$ or $f+2$, as required. This is shown in Table \ref{table:332} below.
		
		\begin{table}[h]
			\setlength{\tabcolsep}{10pt}
			\centering 
			\begin{tabular} {|c c l l l l l|}
				\hline 
				Path & Number & \multicolumn{5}{l|}{Vertex position within the path} \\ 
				sequence & of paths & 1 & 2 to $f-3$ & $f-2$ & $f-1$ & $f$ \\
				\hline 
				& & \multicolumn{5}{l|}{Difference between the pair of vertices} \\
				1 & 1 & $f+1$ & $f+1$ & $f+1$ & $f+1$ & $f+1$ \\
				2 to $f+6$ & $f+5$ & $f+1$ & $f+1$ & $f+1$ & $f+1$ & $f+1$ \\ 
				$f+7$ and $f+8$ & 2 & $f+1$ & $f+1$ & $f+2$ & $f+2$ & $f+1$ \\ 
				$f+9$ to $2f$ & $f-8$ & $f+1$ & $f+1$ & $f+1$ & $f+2$ & $f+1$ \\ 
				$2f+1$ & 1 & $f+1$ & $f+1$ & $f+1$ & $f+2$ & $f+2$ \\ 
				$2f+2$ & 1 & $f+1$ & $f+1$ & $f+1$ & $f+2$ & $f+2$ \\ 
				$2f+3$ & 1 & $f+1$ & $f+1$ & $f+1$ & $f+2$ & 0 \\ 
				\hline 
			\end{tabular}
			\caption{The difference between the pair of vertices at each step of circuits $C_1$ and $C_2$}
			\label{table:332} 
		\end{table}
		
		In order to prove that circuit $C_2$ is Eulerian, it is sufficient to show that every edge of $C_1$ is contained in $C_2$. At the first step, $C_2$ establishes an offset of $f+1$ from $C_1$ and maintains this offset for a major part of the circuit. Therefore it is convenient to compare $C_2$ with a version of $C_1$ that is offset by $f+1$, denoted by $C_1^*$. We identify all the steps where the edges in these two circuits differ (a total of $2f-1$ steps) and confirm that these two sets of edges are identical. A step in each circuit is denoted by its block number, 1 to $2f+3$, and its position within the block, 1 to $f$. An edge in each circuit is denoted by the two vertices that it connects. It is clear from Table \ref{table:331} that the edges that differ at any step occur at positions 1, $f-2$, $f-1$, or $f$ within the blocks and are generated by elements 2, 3, $f-1$ or $f$. These $2f-1$ edges are listed in Table \ref{table:333}, along with their location within the circuits $C_1^*$ and $C_2$.
		\begin{table}[h]
			\setlength{\tabcolsep}{10pt}
			\centering 
			\begin{tabular} {|c c c l l l l|}
				\hline 
				Edge   & \multicolumn{2}{l}{Edge vertices} & \multicolumn{2}{l}{Circuit $C_1^*$} & \multicolumn{2}{l|}{Circuit $C_2$}   \\ 
				generator & from & to & Block & Position & Block & Position \\
				\hline 
				2 & 0 & 2 & $f+7$ & $f-1$ & 1 & 1 \\
				& 1 & 3 & $f+8$ & $f-1$ & $f+7$ & $f-1$ \\
				& 2 & 4 & $f+9$ & $f-1$ & $f+8$ & $f-1$ \\[1ex]
				& \multicolumn{6}{l|}{For $i=3$ to $f-4$:} \\
				& $i$ & $i+2$ & $f+7+i$ & $f-1$ & $f+4+i$ & $f$ \\[1ex]
				3 & $f-3$ & $f$ & $2f+2$ & $f$ & $2f+1$ & $f$ \\
				& $f-2$ & $f+1$ & $2f+3$ & $f$ & $2f+2$ & $f$ \\[1ex]
				& \multicolumn{6}{l|}{For $i=2$ to $f-4$:} \\
				& $i$ & $i+3$ & $f+5+i$ & $f$ & $f+7+i$ & $f-1$ \\[1ex]
				$f-1$ & $f-1$ & $0$ & $2f+2$ & $1$ & $2f+3$ & $f$ \\
				& $f$ & $1$ & $2f+3$ & $1$ & $f+7$ & $f-2$ \\
				& $f+1$ & $2$ & $1$ & $1$ & $f+8$ & $f-2$ \\
				$f$ & $f$ & $0$ & $f+7$ & $f-2$ & $2f+2$ & $1$ \\
				& $f+1$ & $1$ & $f+8$ & $f-2$ & $2f+2$ & $1$ \\
				\hline 
			\end{tabular}
			\caption{The edges that are at different locations within the circuits $C_1^*$ and $C_2$}
			\label{table:333}
		\end{table}
		
		By reference to Table \ref{table:332}, all the other edges occur at the same position within the two circuits. Therefore circuit $C_1^*$ contains all the edges of $C_2$ and so is Eulerian, which means that $C_1$ is also Eulerian. From Table~\ref{table:332}, the circuits $C_1$ and $C_2$ are avoiding.
	\end{proof}
	It remains to consider the case of even orders $n$. We deal separately with the two cases $n\equiv 0\pmod 4$ and $n\equiv 2\pmod 4$, and in each case we give two different proofs of the result to illustrate the techniques involved.
	\begin{theorem}\label{thm:max0mod4}
		For any $n\equiv 0\pmod{4}$ with $n\ge 8$, there exists an edge-maximal doubly Eulerian circulant graph of order $n$ which is regular of degree $n-4$.
	\end{theorem}
	\begin{proof}[Proof 1]
		Let $n=4k$, $k\geq 2$. Our first proof is direct; the circulant graph which we prove to be edge-maximal doubly Eulerian is the complete multipartite graph $K_{4,4,\ldots,4}$ consisting of $k$ independent sets of size 4 which we label $A,B,C,\ldots$. The vertices are labelled $a_0,a_1,a_2,a_3,b_0,b_1,\ldots$ in an obvious manner. All edges exist between sets $X$ and $Y$, $X,Y\in\{A,B,C,\ldots\}$, $X\neq Y$.
		
		We adopt a similar strategy to that in the proof of Theorem~\ref{thm:3mod6}. Remove the edges $a_0\edge b_0$, $a_0\edge b_1$, $a_3\edge b_0$ and $a_3\edge b_1$. The reduced graph remains Eulerian. Choose an Eulerian circuit $E$ beginning and ending at the vertex $b_1$. Extend this to an Eulerian circuit of $K_{4,4,\ldots,4}$ as follows:
		
		$\quad a_0,b_0,a_3,b_1,E,b_1,a_0$.

  \newpage
		Now construct a second ``parallel'' Eulerian circuit
		
		$\quad a_1,b_1,a_0,b_2,E^{(+1)},b_2,a_1$,
		
		where if the first circuit visits vertex $x_i$ at a particular step, the second visits $x_{i+1}$ (arithmetic modulo 4). Now replace $a_1$ with $a_0$ at the start and end points of the second ``parallel'' circuit and replace the first visit to $a_0$ with $a_1$. The result is a pair of avoiding Eulerian circuits, beginning and ending at $a_0$.
	\end{proof}
	\begin{proof}[Proof 2]
		Our second proof proceeds by induction. The result is true for the case $n=8$ since we have a pair of avoiding Eulerian circuits on the graph $K_{4,4}$ (see for example Table~\ref{tab:small} or Proof 1). Let $n=4k$, $k\geq 3$ and suppose we have a pair of avoiding Eulerian circuits on the graph $K_{4,4,\ldots,4}$ of order $4(k-1)$, with vertices labelled as in Proof 1. We now extend to the larger graph of order $4k$ by adjoining a new independent set $Z$ with vertices $z_0,z_1,z_2,z_3$.
		
		In the first circuit choose a vertex $a_0$, \emph{other than at the beginning or end of the circuit}. Replace $a_0$ by an Eulerian circuit on the $A$ and $Z$ partitions beginning and ending at $a_0$. In the second circuit, \emph{in the same place}, there will be vertex $a_i$ where $i\neq 0$. Replace $a_i$ by the same Eulerian circuit on the $A$ and $Z$ partitions as above, but to every vertex add $i$ to its index (arithmetic modulo 4). Now do the same for the vertex $b_0$ on the $B$ and $Z$ partitions, the vertex $c_0$ on the $C$ and $Z$ partitions and so on. The result is a pair of avoiding circuits beginning and ending at $a_0$.
	\end{proof}
	\begin{observation}
		We observed earlier that the complete bipartite graph $K_{4,4}$ has avoidance index 3. The second proof of the above theorem can be extended to show that the complete multipartite graph $K_{4,4,\ldots,4}$ has avoidance index at least 3. In fact it is exactly 3, which follows from a more general result which we prove in Section~\ref{sec:index}.
	\end{observation}
	
	We now turn our attention to the remaining case where $n\equiv 2\pmod 4$.
	\begin{theorem}\label{thm:max2mod4}
		For any $n\equiv 2\pmod{4}$ with $n\geq 10$, there exists an edge-maximal doubly Eulerian circulant graph of order $n$ which is regular of degree $n-4$.
	\end{theorem}
	Again we offer two proofs but, unlike in the previous theorem, here the graphs which we prove to be edge-maximal doubly Eulerian are different. In the first proof the graph is not a circulant; in the second it is.
	\begin{proof}[Proof 1]
		We follow the strategy of the first proof of Theorem~\ref{thm:max0mod4}. First consider the case $n=10$. Let $G$ be the graph with vertex set $\{i:0\leq i\leq 9\}$. For all $i$, $0\leq i\leq 3$ and all $j$, $4\leq j\leq 9$, vertex $i$ is connected to vertex $j$. Further edges are $4\edge 5$, $5\edge 6$, $6\edge 7$, $7\edge 8$, $8\edge 9$ and $9\edge 4$. We need to exhibit two pairs of mutually avoiding Eulerian circuits, one starting and ending at the vertex 0 and the other at the vertex 4.
		
		Vertex 0 circuit $C_1$: 0	4	1	5	2	6	3	7	0	8	1	9	2	4	5	6	7	8	9	4	3	5	0	6	1	7	2	8	3	9	0
		
		Vertex 0 circuit $C_2$: 0	6	2	7	0	8	1	9	3	4	2	5	3	6	7	8	9	4	5	6	1	7	3	8	2	9	0	4	1	5	0
		
		Vertex 4 circuit $C_1$: 4	1	5	2	6	3	7	0	8	1	9	2	4	5	6	7	8	9	4	3	5	0	6	1	7	2	8	3	9	0	4
		
		Vertex 4 circuit $C_2$: 4	2	7	1	8	0	9	3	4	0	5	3	6	7	8	9	4	5	6	0	7	3	8	2	9	1	6	2	5	1	4
		
		Now let $n=4k+2$, $k\geq 3$ and consider the graph $K_{4,4,\ldots,4,6}$ of order $4k+2$. This graph consists of $k-1$ independent sets of size 4, which we label $A,B,C,\ldots$, together with one further set $Z$ of size 6. The vertices are labelled $a_0,a_1,a_2,a_3,b_0,b_1,\ldots$ and $z_0,z_!,\ldots,z_5$ in an obvious manner. All edges exist between sets $X$ and $Y$, $X,Y\in\{A,B,C,\ldots,Z\}$, $X\neq Y$. In addition, to achieve an $(n-4)$-regular graph we include a 6-cycle $z_0,z_1,\ldots,z_5,z_0$.
		
		As before, we construct one Eulerian circuit on this graph starting from vertex $a_0$ with initial vertices $a_0,b_0,a_3,b_1,\ldots$ and ending with $\ldots,b_1,a_0$. Now construct a second `parallel' Eulerian circuit beginning $a_1,b_1,a_0,b_2,\ldots$ and ending $\ldots,b_2,a_1$. This time if the first circuit visits vertex $x_i$ at a particular step, where $x_i$ is in one of the sets $A,B,C,\ldots$ of size 4, the second visits $x_{i+1}$ (arithmetic modulo 4). But if the first circuit visits $z_i$, then the second visits $z_{i+2}$ (arithmetic modulo 6). As before, replace $a_1$ with $a_0$ as the start and end point of this new circuit, and replace the first visit to $a_0$ with $a_1$; the result is a pair of avoiding Eulerian circuits, starting and ending at $a_0$. 
		
		It remains to find a pair of such circuits starting and ending at $z_0$. For the first circuit, we choose one starting $z_0,a_0,z_4,a_1,\ldots$ and ending $\ldots,a_1,z_0$. The parallel circuit, using the same rules as above, begins $z_2,a_1,z_0,a_2,\ldots$ and ends $\ldots,a_2,z_2$. Now replace $z_2$ with $z_0$ as the start and end point of this new circuit, and replace the first visit to $z_0$ with $z_2$; the result is a pair of avoiding Eulerian circuits, starting and ending at $z_0$.
	\end{proof}
	For our second proof we present a doubling construction.
	\begin{proof}[Proof 2]
		Let $G$ be the graph obtained from the complete graph $K_{2m+1}$ on vertex set $\{0,1,2,\ldots,2m\}$ by removing the Hamiltonian cycle given by the edges $i\edge(i+1)$ (arithmetic modulo $2m+1$). By Theorem~\ref{thm:oddncirc}, for $m\geq 4$ the graph admits a pair of mutually avoiding Eulerian circuits $C_1$ and $C_2$, without loss of generality starting and ending at vertex 0. Represent $C_1$ as $0\ C_1^-\ 0\ C_1^+\ 0$, where the intermediate occurrence of 0 is the first time that the circuit returns to this vertex. Represent $C_2$ as $0\ C_2^-\ z\ C_2^+\ 0$, where $z$ is the vertex in $C_2$ when $C_1$ is at the first occurrence of 0. In fact, $z=\pm 1$.
		
		Now let $G'$ be an isomorphic copy of $G$ on the vertex set $\{0',1',2',\ldots,(2m)'\}$. Introduce connections between the graphs $G$ and $G'$ by adjoining all edges $x\edge y'$, $x\neq y$. The result is a circulant graph of order $4m+2$ and degree $4m-2$. Let $\hat{C}$ be the circuit (not Eulerian) obtained by replacing every edge $x\edge y$ in $C_1$ by the path $x\ y'\ x'\ y$, and further let $\hat{C}^{(z)}$ be the circuit offset by $z$, i.e. $x\edge y$ in $C_1$ is replaced by $(x+z)\ (y+z)'\ (x+z)'\ (y+z)$. The following are a pair of mutually avoiding Eulerian circuits.
		
		$\quad 0,C_1^-,0,\hat{C},0,1',2,3',\ldots,(2m),0',1,2',\ldots,(2m)',0,C_1^+,0$; and
		
		$\quad 0,C_2^-,z,\hat{C}^{(z)},z,(z+1)',(z+2),(z+3)',\ldots,(z+2m),z',(z+1),(z+2)',\ldots,(z+2m)',z,C_2^+,0$.
		
		It remains to deal separately with the two cases where $m=2$ or $m=3$. Pairs of mutually avoiding Eulerian circuits for these two cases are as follows.
		
		$4m+2=10$:
		
		Circuit $C_1$: 0\ 4'\ 3\ 2'\ 1\ 0'\ 4\ 3'\ 2\ 1'\ 0\ 3\ 1'\ 3'\ 0'\ 3\ 1\ 4'\ 2\ 4\ 2'\ 4'\ 1'\ 4\ 1\ 3'\ 0\ 2\ 0'\ 2'\ 0
		
		Circuit $C_2$: 0\ 3'\ 2\ 1'\ 0\ 4'\ 3\ 2'\ 1\ 0'\ 4\ 3'\ 0'\ 3\ 1'\ 4\ 2\ 0'\ 2'\ 4'\ 1'\ 3'\ 1\ 3\ 0\ 2'\ 4\ 1\ 4'\ 2\ 0

		$4m+2=14$:
		
		Circuit $C_1$: 0\ 6'\ 5\ 4'\ 3\ 2'\ 1\ 0'\ 6\ 5'\ 4\ 3'\ 2\ 1'\ 0\ 5\ 3\ 1\ 6\ 4\ 2\ 0\ 4'\ 2'\ 0'\ 5'\ 3'\ 1'\ 6'\ 4'\ 1\ 5'\ 2\ 6'\ 3\ 0'\ 4\ 1'\ 5\ 2'\ 6\ 3'\ 0\ 2'\ 5'\ 1'\ 4'\ 0'\ 5\ 1\ 6'\ 2'\ 4\ 0\ 5'\ 3\ 6\ 2\ 5\ 3'\ 6'\ 4\ 1\ 3'\ 0'\ 2\ 4'\ 6\ 1'\ 3\ 0
		
		Circuit $C_2$: 0\ 5'\ 4\ 3'\ 2\ 1'\ 0\ 6'\ 5\ 4'\ 3\ 2'\ 1\ 0'\ 6\ 5'\ 3'\ 1'\ 6'\ 4'\ 2'\ 0'\ 5'\ 2\ 0\ 5\ 3\ 1\ 6\ 4\ 2\ 6'\ 3\ 0'\ 4\ 1'\ 5\ 2'\ 6\ 3'\ 0\ 4'\ 1\ 3'\ 6'\ 2'\ 5'\ 1'\ 6\ 2\ 0'\ 3'\ 5\ 0'\ 4'\ 2\ 5\ 1\ 5'\ 3\ 6\ 4'\ 1'\ 3\ 0\ 2'\ 4\ 6'\ 1\ 4\ 0
		
	\end{proof}
	\begin{example}
		We illustrate Proof 2 of the preceding theorem by an example. Let $m=4$ and let $G$ be the complete graph $K_9$ on vertex set $\{0,1,\ldots,8\}$, with the Hamiltonian cycle given by the edges $i\edge(i+1)$ (arithmetic modulo 9) removed. A pair of mutually avoiding circuits on $G$ beginning and ending at 0 are the following.
		
		Circuit $C_1$: 0\ 7\ 5\ 3\ 1\ 8\ 6\ 4\ 2\ 0\ 4\ 8\ 3\ 6\ 2\ 5\ 8\ 2\ 7\ 4\ 1\ 6\ 0\ 5\ 1\ 7\ 3\ 0
		
		Circuit $C_2$: 0\ 6\ 4\ 2\ 0\ 7\ 5\ 3\ 1\ 8\ 3\ 7\ 2\ 5\ 1\ 4\ 7\ 1\ 6\ 3\ 0\ 5\ 8\ 6\ 2\ 8\ 4\ 0
		
		Applying the doubling construction employed in the proof yields a pair of mutually avoiding Eulerian circuits on a 14-regular graph of order 18 as follows.
		
		Circuit $C_1$: 0\ 7\ 5\ 3\ 1\ 8\ 6\ 4\ 2\ 0\ 7'\ 0'\ 7\ 5'\ 7'\ 5\ 3'\ 5'\ 3\ 1'\ 3'\ 1\ 8'\ 1'\ 8\ 6'\ 8'\ 6\ 4'\ 6'\ 4\ 2'\ 4'\ 2\ 0'\ 2'\ 0\ 4'\ 0'\ 4\ 8'\ 4'\ 8\ 3'\ 8'\ 3\ 6'\ 3'\ 6\ 2'\ 6'\ 2\ 5'\ 2'\ 5\ 8'\ 5'\ 8\ 2'\ 8'\ 2\ 7'\ 2'\ 7\ 4'\ 7'\ 4\ 1'\ 4'\ 1\ 6'\ 1'\ 6\ 0'\ 6'\ 0\ 5'\ 0'\ 5\ 1'\ 5'\ 1\ 7'\ 1'\ 7\ 3'\ 7'\ 3\ 0'\ 3'\ 0\ 1'\ 2\ 3'\ 4\ 5'\ 6\ 7'\ 8\ 0'\ 1\ 2'\ 3\ 4'\ 5\ 6'\ 7\ 8'\ 0\ 4\ 8\ 3\ 6\ 2\ 5\ 8\ 2\ 7\ 4\ 1\ 6\ 0\ 5\ 1\ 7\ 3\ 0
		
		Circuit $C_2$: 0\ 6\ 4\ 2\ 0\ 7\ 5\ 3\ 1\ 8\ 6'\ 8'\ 6\ 4'\ 6'\ 4\ 2'\ 4'\ 2\ 0'\ 2'\ 0\ 7'\ 0'\ 7\ 5'\ 7'\ 5\ 3'\ 5'\ 3\ 1'\ 3'\ 1\ 8'\ 1'\ 8\ 3'\ 8'\ 3\ 7'\ 3'\ 7\ 2'\ 7'\ 2\ 5'\ 2'\ 5\ 1'\ 5'\ 1\ 4'\ 1'\ 4\ 7'\ 4'\ 7\ 1'\ 7'\ 1\ 6'\ 1'\ 6\ 3'\ 6'\ 3\ 0'\ 3'\ 0\ 5'\ 0'\ 5\ 8'\ 5'\ 8\ 4'\ 8'\ 4\ 0'\ 4'\ 0\ 6'\ 0'\ 6\ 2'\ 6'\ 2\ 8'\ 2'\ 8\ 0'\ 1\ 2'\ 3\ 4'\ 5\ 6'\ 7\ 8'\ 0\ 1'\ 2\ 3'\ 4\ 5'\ 6\ 7'\ 8\ 3\ 7\ 2\ 5\ 1\ 4\ 7\ 1\ 6\ 3\ 0\ 5\ 8\ 6\ 2\ 8\ 4\ 0
	\end{example}
	\section{Bipartite edge-maximal doubly Eulerian graphs}\label{sec:bipartite}
	In Section~\ref{sec:maximal}, it was shown that there exist doubly Eulerian graphs which attain the upper density bound. In this section we restrict our attention to edge-maximal bipartite graphs. We have the following result.
	
	\begin{lemma}\label{lem:bipgeneral}
		Let $G$ be a bipartite Eulerian graph, and suppose $G$ admits a subgraph $K\cong K_{3,2}$ such that removing the edges of $K$ does not disconnect $G$. Let $v$ be a vertex of $G$ which is in the larger partition of $K$. Then $G$ admits a pair of avoiding Eulerian circuits starting and ending at $v$.
	\end{lemma}
	\begin{proof}
		Let the vertices of $K$ be $t,u,v$ in one partition and $w,x$ in the other. Remove the edges of $K$ from $G$. The result is a connected graph with every edge having even valency except $w,x$. So there exists an Eulerian trail $T$ in the graph from $w$ to $x$. Two avoiding Eulerian circuits in $G$ can now be constructed as follows.
		
		$\quad v,x,u,w,T,x,t,w,v$;
		
		$\quad v,w,t,x,u,w,T,x,v$.
	\end{proof}
	\begin{theorem}\label{thm:bipeven}
		The complete bipartite graph $K_{2r,2s}$, $r,s \geq 2$, is doubly Eulerian.
	\end{theorem}
	\begin{proof}
		Every vertex of $K_{2r,2s}$ satisfies the conditions of Lemma~\ref{lem:bipgeneral}, so we can construct two avoiding Eulerian circuits starting and ending at any vertex.
	\end{proof}
	
	We remark that complete bipartite graphs $K_{2,2s}$, $s \geq 1$, are not doubly Eulerian. Two circuits both starting at either vertex of the 2-partition cannot be avoiding. We have one further result. Denote by $K^*_{2r+1,2r+1}$, $r \geq 1$, the complete bipartite graph minus a perfect matching. 

	\begin{theorem}\label{thm:bipodd}
		The graph $K^*_{2r+1,2r+1}$, $r \geq 2$, is doubly Eulerian.
	\end{theorem}
	\begin{proof}
		Every vertex of the graph satisfies the conditions of Lemma~\ref{lem:bipgeneral}.
	\end{proof}
	
	Finally, again we remark that the graph $K^*_{3,3}$ is the cycle $C_6$ and is not doubly Eulerian.

	\section{Edge-minimal doubly Eulerian graphs}\label{sec:minimal}
	Define the \emph{edge excess} of a graph $G=(V,E)$ to be $\xi = \left| E \right| - \left| V \right|$. If $\xi = 0$, then we have the cycles $C_n$, $n \geq 3$ which are not doubly Eulerian. If $\xi = 1$ then for an Eulerian graph G, one vertex has valency 4 and all others have valency 2. An example is illustrated in Figure~\ref{fig:xs12}(a).  It is immediate that such graphs cannot be doubly Eulerian; consider a starting vertex at maximum distance from the valency 4 vertex on either of the cycles. If $\xi = 2$ then for an Eulerian graph there are two possibilities: (1) one vertex has valency 6 and all others have valency 2, or (2) two vertices have valency 4 and all others have valency 2. Again graphs satisfying possibility (1) cannot be doubly Eulerian and an example is given in Figure~\ref{fig:xs12}(b).
	
	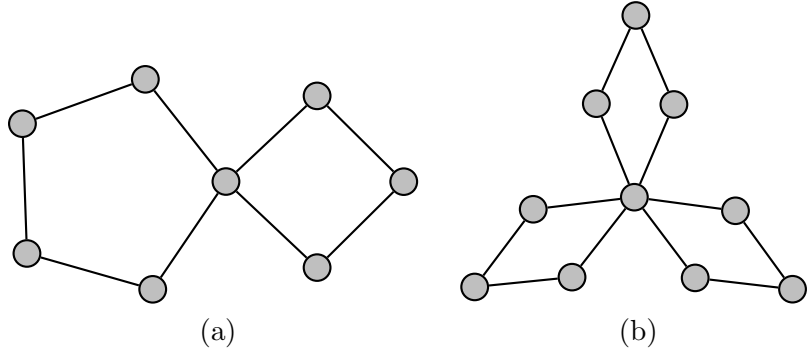
\begin{figure}
		\centering
		\begin{tabular}{cc}
			\begin{tikzpicture}[x=0.2mm,y=-0.2mm,inner sep=0.2mm,scale=0.9,thick,vertex/.style={circle,draw,minimum size=10,fill=lightgray}]
	\node at (75.3,-1.9) [vertex] (v1) {};
	\node at (16.4,-77.2) [vertex] (v2) {};
	\node at (-73.4,-44.5) [vertex] (v3) {};
	\node at (-70.1,51.1) [vertex] (v4) {};
	\node at (21.8,77.5) [vertex] (v5) {};
	\node at (141.9,61.5) [vertex] (v6) {};
	\node at (205.2,-1.8) [vertex] (v7) {};
	\node at (141.9,-65) [vertex] (v8) {};
	\draw (v1) to (v2);
	\draw (v2) to (v3);
	\draw (v3) to (v4);
	\draw (v4) to (v5);
	\draw (v5) to (v1);
	\draw (v1) to (v6);
	\draw (v6) to (v7);
	\draw (v7) to (v8);
	\draw (v1) to (v8);
\end{tikzpicture} & \begin{tikzpicture}[x=0.2mm,y=-0.2mm,inner sep=0.2mm,scale=0.5,thick,vertex/.style={circle,draw,minimum size=10,fill=lightgray}]
\node at (288,476) [vertex] (v1) {};
\node at (370,370) [vertex] (v2) {};
\node at (237,386) [vertex] (v3) {};
\node at (503,388) [vertex] (v4) {};
\node at (578,492) [vertex] (v5) {};
\node at (160,489) [vertex] (v6) {};
\node at (372,129) [vertex] (v7) {};
\node at (422,247) [vertex] (v8) {};
\node at (320,246) [vertex] (v9) {};
\node at (450,477) [vertex] (v10) {};
\path
	(v1) edge (v2)
	(v2) edge (v3)
	(v6) edge (v1)
	(v2) edge (v9)
	(v7) edge (v9)
	(v7) edge (v8)
	(v2) edge (v8)
	(v3) edge (v6)
	(v5) edge (v10)
	(v4) edge (v5)
	(v2) edge (v10)
	(v2) edge (v4)
	;
\end{tikzpicture} \\
			(a) & (b)
		\end{tabular}
		\caption{Graphs which cannot be doubly Eulerian}
		\label{fig:xs12}
	\end{figure}
	
	For possibility (2), the vertices of valency 2 form four paths between the two vertices of valency 4 and so these graphs are uniquely determined by a set of four parameters $0 \leq a \leq b \leq c \leq d$ for the number of valency 2 vertices in each path A, B, C, D respectively. It will also be convenient to regard these paths as ordered sets. Denote the valency 4 vertices by $u$ and $v$ and the graphs by the notation $\Gamma = \Gamma(V,E) = \Gamma(a,b,c,d)$. An example is shown in Figure~\ref{fig:Gamma_abcd}. Let $\rho(x,y)$ be the distance between two vertices $x,y \in V$. These graphs have avoidance index of either 1 or 2. 
	
	\begin{figure}\centering
		\begin{tikzpicture}[x=0.2mm,y=0.2mm,thick,vertex/.style={circle,draw,minimum size=10,inner sep=0,fill=lightgray}]
	\node at (-130,-10) [vertex,label={left:$u$}] (v1) {};
	\node at (110,-10) [vertex,label={right:$v$}] (v2) {};
	\node at (-10,50) [vertex,label={below:$a_1$}] (v3) {};
	\node at (-10,10) [vertex,label={below:$b_1$}] (v4) {};
	\node at (-50,-30) [vertex,label={below:$c_1$}] (v5) {};
	\node at (30,-30) [vertex,label={below:$c_2$}] (v6) {};
	\node at (-70,-70) [vertex,label={below:$d_1$}] (v7) {};
	\node at (-10,-70) [vertex,label={below:$d_2$}] (v8) {};
	\node at (50,-70) [vertex,label={below:$d_3$}] (v9) {};
	\draw (v1) to (v3);
	\draw (v2) to (v3);
	\draw (v1) to (v4);
	\draw (v2) to (v4);
	\draw (v1) to (v5);
	\draw (v5) to (v6);
	\draw (v2) to (v6);
	\draw (v1) to (v7);
	\draw (v7) to (v8);
	\draw (v8) to (v9);
	\draw (v2) to (v9);
\end{tikzpicture}
		\caption{The graph $\Gamma(1,1,2,3)$}
		\label{fig:Gamma_abcd}
	\end{figure}
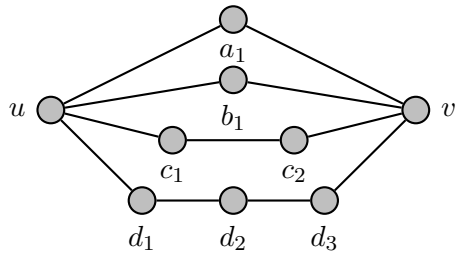
	
	The smallest interesting case is graphs of order 8. From Table~\ref{tab:small} we know that there are 7 doubly Eulerian graphs of this order; the 5 graphs which are 4-regular are edge-maximal and were discussed in Section~\ref{sec:maximal}. The remaining graphs are shown in Figure~\ref{fig:minimal_8}. Graph (a) has edge excess $\xi=2$ and is easily seen to be $\Gamma(0,2,2,2)$ as described above; graph (b) has edge excess $\xi=4$.
	
	\begin{figure}
		\centering
		\begin{tabular}{cc}
			\begin{tikzpicture}[x=0.2mm,y=-0.2mm,inner sep=0.2mm,scale=0.9,thick,vertex/.style={circle,draw,minimum size=10,fill=lightgray}]
\node at (238,310) [vertex] (v1) {};
\node at (334,310) [vertex] (v2) {};
\node at (382,310) [vertex] (v3) {};
\node at (238,358) [vertex] (v4) {};
\node at (334,358) [vertex] (v5) {};
\node at (382,358) [vertex] (v6) {};
\node at (286,262) [vertex] (v7) {};
\node at (286,406) [vertex] (v8) {};
\path
	(v1) edge (v4)
	(v1) edge (v7)
	(v2) edge (v5)
	(v2) edge (v7)
	(v3) edge (v6)
	(v3) edge (v7)
	(v4) edge (v8)
	(v5) edge (v8)
	(v6) edge (v8)
	(v7) edge (v8)
	;
\end{tikzpicture} & \begin{tikzpicture}[x=0.2mm,y=-0.2mm,inner sep=0.2mm,scale=0.8,thick,vertex/.style={circle,draw,minimum size=10,fill=lightgray}]
\node at (260,300) [vertex] (v1) {};
\node at (260,420) [vertex] (v2) {};
\node at (380,360) [vertex] (v3) {};
\node at (440,360) [vertex] (v4) {};
\node at (200,360) [vertex] (v5) {};
\node at (260,360) [vertex] (v6) {};
\node at (380,420) [vertex] (v7) {};
\node at (380,300) [vertex] (v8) {};
\path
	(v1) edge (v5)
	(v1) edge (v6)
	(v1) edge (v7)
	(v1) edge (v8)
	(v2) edge (v5)
	(v2) edge (v6)
	(v2) edge (v7)
	(v2) edge (v8)
	(v3) edge (v7)
	(v3) edge (v8)
	(v4) edge (v7)
	(v4) edge (v8)
	;
\end{tikzpicture} \\
			(a) & (b)
		\end{tabular}
		
		\caption{The graphs of order 8 with avoidance index 2 which are not regular}
		\label{fig:minimal_8}
	\end{figure}
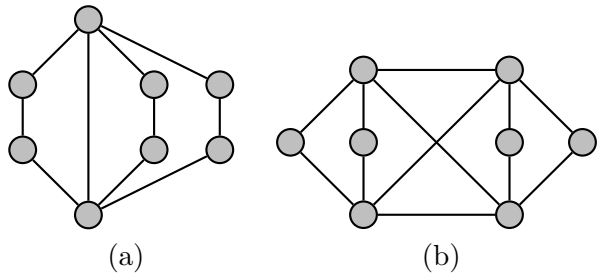
	
	Our investigations suggest that a complete analysis of the values of $a, b, c, d$ for which the graph $\Gamma(a,b,c,d)$ has either index is extremely complex and tedious and breaks into many different cases and subcases. We therefore restrict our attention to the case where $a=0$, i.e. the class of graphs for which the two vertices of valency 4 are adjacent. Two results are easily established.
	
	\begin{lemma}\label{abc3}
		If $d \geq a+b+c+3$ then the graph $\Gamma$ has avoidance index 1.
	\end{lemma}
	\begin{proof}
		Suppose that $d \geq a+b+c+3$. The four paths A, B, C, D contain $a+1$, $b+1$, $c+1$, $d+1$ edges respectively. So the graph $\Gamma$ contains a total number of edges $\left| E \right| = a+b+c+d+4 \leq 2d+1$. Now consider two Eulerian circuits starting on path D at the 2-valent vertex adjacent to vertex $u$. The circuit starting along the path D in the direction of the vertex $v$ includes $d$ edges to reach $v$ and $d+1$ edges to clear it. Conversely, the final $d$ edges of the other circuit start at vertex $v$ and are along path D, with the preceding edge incident to vertex $v$, also $d+1$ in total. so the circuits can only be avoiding if the length of a circuit is at least $2(d+1)$. As $\left| E \right|  \leq 2d+1$, these circuits are not avoiding and so graph $\Gamma$ has avoidance index 1.
	\end{proof}
	
	\begin{lemma}\label{ca1}
		If $c \leq a+1$ then the graph $\Gamma$ has avoidance index 1.
	\end{lemma}
	\begin{proof}
		If $d=a$ or $d=a+1$, then any two Eulerian circuits starting at vertex $u$ first reach vertex $v$ within one step of each other and so are not avoiding.\\
		Now suppose that $d>a+1$. Let $\mathcal{C}_1$ and $\mathcal{C}_2$ be two Eulerian circuits starting at the vertex on path D distant $\lfloor(d-a)/2\rfloor$ from vertex $u$, with $\mathcal{C}_1$ in the direction of $u$ and $\mathcal{C}_2$ in the direction of $v$. Then $\mathcal{C}_2$ reaches vertex $v$ after $d+1-\lfloor(d-a)/2\rfloor= \lceil(d+a)/2\rceil+1$ steps. Now $\mathcal{C}_1$  reaches vertex $u$ after $\lfloor(d-a)/2\rfloor$ steps. It continues from $u$ to $v$ along a path of length $a+1$ or $a+2$ reaching $v$ after a total number of steps either $\lfloor(d+a)/2\rfloor+1$ or $\lfloor(d+a)/2\rfloor+2$. This is always within one step of $\mathcal{C}_2$. Hence the two circuits are not avoiding. 
	\end{proof}
	
	The next theorem, together with the above Lemma \ref{abc3}, provides the complete analysis of the case where $a=0$.
	
	\begin{theorem}\label{a0}
		The graph $\Gamma(a,b,c,d)$ where $a=0$ and $d < a+b+c+3$ has avoidance index 2 if and only if it is bipartite.
	\end{theorem}
	\begin{proof}
		First suppose that $\Gamma$ is not bipartite. Then as $a=0$, at least one of the paths B, C, D between vertices $u$ and $v$ must have an odd number of valency 2 vertices. Denote such a path by X. Now consider two circuits starting from the middle vertex of X in opposite directions. They will arrive at vertices $u$ and $v$ after the same number of steps. But $\rho(u,v)=1$, and so these circuits are not avoiding. Hence $\Gamma$ has avoidance index 1.
		
		Now suppose that $\Gamma$ is bipartite. Then  as $a=0$, the paths B, C, D between vertices $u$ and $v$ also have an even number of valency 2 vertices and as $\Gamma$ is simple, $b,c,d \geq 2$.
		
		Consider two Eulerian circuits $\mathcal{C}_1$ and $\mathcal{C}_2$ starting at a common vertex. For $j = 1,2$, define  mappings $T_j$ from the set $\{0,1,\ldots,\left| E \right|\}$ to $V$ such that for $0 \leq i \leq \left| E \right|$, $T_j(i)$ is the vertex in the circuit $j$ reached after $i$ steps. Since $\Gamma$ is bipartite, $T_1(i)$ and $T_2(i)$ will lie in the same partition, so that $\rho(T_1(i), T_2(i))$ is even. Hence if $T_1(i)$ and $T_2(i)$ are distinct then  $\rho(T_1(i),T_2(i)) > 1$. Thus it suffices to show, for every vertex, the existence of a pair of Eulerian circuits $\mathcal{C}_1$, $\mathcal{C}_2$, starting at that vertex such that $T_1(i)$ and $T_2(i)$ are distinct for every $i:0<i<\left| E \right| $.
		
		First consider Eulerian circuits $\mathcal{C}_1$ and $\mathcal{C}_2$ starting at vertex $u$. Suppose that $\mathcal{C}_1$ follows path D to vertex $v$, then path C to vertex $u$, path B to vertex $v$ and finally path A to return to vertex $u$. Meanwhile $\mathcal{C}_2$ follows path A to vertex $v$, then path B to vertex $u$, path D to vertex $v$ and finally path C to return to vertex $u$. Schematically these circuits can be represented by $u$D$v$C$u$B$vu$ and $uv$B$u$D$v$C$u$. First consider the vertices $u$ and $v$. Now $T_1(i)=u$ for $i=d+c+2$ and $T_2(i)=u$ for $i=b+2$. As $d+c > b$, there is no conflict at vertex $u$. For vertex $v$, we have $T_1(i)=v$ for $i=d+1$ and $i=d+c+b+3$, and $T_2(i)=v$ for $i=1$ and $i=b+d+3$. As $1 < d+1 < b+d+3 < d+c+b+3$, there is also no conflict at vertex $v$. Both circuits traverse path C from vertex $v$ and path D from vertex $u$ in the same direction. Since $b \leq c \leq d$, $T_1(i) \in \text{B} \implies T_2(i) \in \text{C} \cup\{v\}$ and $T_2(i) \in \text{B} \implies T_1(i) \in \text{D} \cup\{v\}$. Thus there are no conflicts on any of the paths and hence the two circuits $\mathcal{C}_1$ and $\mathcal{C}_2$ are avoiding.
		
		Now consider Eulerian circuits $\mathcal{C}_1$ and $\mathcal{C}_2$ starting at vertex $s$, an arbitrary valency 2 vertex on path X, where X is one of the paths B, C or D. Let the other two paths of B, C or D be denoted by Y and Z, and let the number of valency 2 vertices in X, Y and Z be $x$, $y$ and $z$ respectively with $y \leq z$. Let $\textnormal{X}_u$ and $\textnormal{X}_v$ be the subpaths from $s$ to $u$ and $v$ respectively, with $x_u$ and $x_v$ valency 2 vertices apart from vertex $s$, so that $x_u+x_v+1=x$. Recall that $x$ is even. Without loss of generality, assume that $x_u<x_v$.
		
		There are two cases to consider. First, the case $x_u+1<x_v$. Suppose that $\mathcal{C}_1$ follows path $\textnormal{X}_u$ to vertex $u$, then path A to vertex $v$, path Z to vertex $u$, path Y to vertex $v$ and finally path $\textnormal{X}_v$ to return to vertex $s$. Meanwhile $\mathcal{C}_2$ follows path $\textnormal{X}_v$ to vertex $v$, then path Z to vertex $u$, path Y to vertex $v$, path A to vertex $u$ and finally path $\textnormal{X}_u$ to return to vertex $s$. Schematically these circuits can be represented by $s\textnormal{X}_uuv\textnormal{Z}u\textnormal{Y}v\textnormal{X}_vs$ and
		$s\textnormal{X}_vv\textnormal{Z}u\textnormal{Y}vu\textnormal{X}_us$.
		Consider the vertices $u$ and $v$. Now $T_1(i)=u$ for $i=x_u+1$ and $i=x_u+z+3$ and $T_2(i)=u$ for $i=x_v+z+2$ and $i=x_v+z+y+4$. As $x_u+1 < x_u+z+3 < x_v+z+2 < x_v+z+y+4$, there is no conflict at vertex $u$. For vertex $v$, we have $T_1(i)=v$ for $i=x_u+2$ and $i=x_u+z+y+4$, and $T_2(i)=v$ for $i=x_v+1$ and $i=x_v+z+y+3$. Now since $d<a+b+c+3$, it follows that $x_v < d < b+c+3 \leq y+z+3 $. Therefore $x_u+2 < x_v+1 < x_u+z+y+4 < x_v+z+y+3$, and so there is also no conflict at vertex $v$. 
		Both circuits traverse path Y from vertex $u$ and path Z from vertex $v$ in the same direction. For the path $\textnormal{X}_u$, we have that  $T_1(i) \in \textnormal{X}_u$ for $1 \leq i \leq x_u$ and $T_2(i) \in \textnormal{X}_u$ for $\left| E \right| - x_u \leq i \leq \left| E \right| - 1$. Since $2x_u < x < \left| E \right| $, these two intervals are disjoint. A similar argument applies for path $\textnormal{X}_v$. Since $2x_v<2d-2<\lvert E\rvert$, again the two intervals are disjoint. Thus there are no conflicts on any of the paths and hence the two circuits $\mathcal{C}_1$ and $\mathcal{C}_2$ are avoiding.
		
		The second case is where $x_u+1=x_v$. In this case the two paths $\textnormal{X}_u$ and $\textnormal{X}_v$ are interchanged in the circuits $\mathcal{C}_1$ and $\mathcal{C}_2$ defined above to give circuits
		$s\textnormal{X}_vvu\textnormal{Z}v\textnormal{Y}u\textnormal{X}_us$ and
		$s\textnormal{X}_uu\textnormal{Z}v\textnormal{Y}uv\textnormal{X}_vs$. 	
		Again consider the vertices $u$ and $v$. Now $T_1(i)=u$ for $i=x_v+2$ and $i=x_v+z+y+4$ and $T_2(i)=u$ for $i=x_u+1$ and $i=x_u+z+y+3$. As $x_u+1 < x_v+2 < x_u+z+y+3 < x_v+z+y+4$, there is no conflict at vertex $u$. For vertex $v$, we have $T_1(i)=v$ for $i=x_v+1$ and $i=x_v+z+3$, and $T_2(i)=v$ for $i=x_u+z+2$ and $i=x_u+z+y+4$. Since $x_v+1 < x_u+z+2 < x_v+z+3 < x_u+z+y+4$, there is also no conflict at vertex $v$. The remainder of the proof is as in the case where $x_u+1<x_v$. Hence the two circuits $\mathcal{C}_1$ and $\mathcal{C}_2$ are avoiding.
	\end{proof}
	
	The above theorem gives a construction for edge-minimal doubly Eulerian graphs for all even orders $n \geq 8$. The graphs have edge excess equal to 2. For $n=10$, in addition to the graph $\Gamma(0,2,2,4)$ given by the above theorem, we find that there is a further doubly Eulerian graph with edge excess equal to 2, namely $\Gamma(1,1,3,3)$, also bipartite. There are three edge-minimal doubly Eulerian graphs on 12 vertices, $\Gamma(0,2,2,6)$, $\Gamma(0,2,4,4)$ and $\Gamma(1,3,3,3)$, again all bipartite. The smallest non-bipartite doubly Eulerian graph of even order and with edge excess equal to 2 occurs when $n=14$. There are two such graphs, $\Gamma(1,2,4,5)$ and $\Gamma(1,3,4,4)$. Bipartite graphs are $\Gamma(0,2,4,6)$, $\Gamma(0,4,4,4)$, $\Gamma(1,3,3,5)$ and $\Gamma(2,2,4,4)$. It remains to consider graphs of odd order.
	
	\begin{theorem}\label{a1b4}
		The graph $\Gamma(1,4,c,d)$ where $d=c \geq 4$ has avoidance index 2.
	\end{theorem}
	\begin{proof}
		Counting from vertex $u$ to vertex $v$, let the valency 2 vertices on paths A, B, C, D be $a$, $b_1$, $b_2$, $b_3$, $b_4$, $c_1, \ldots, c_m$, $d_1, \ldots, d_m$. By symmetry, we only need to consider the cases with Eulerian circuits starting at $u$, $a$, $b_i,~i=1,2$ and $c_i,~1 \leq i \leq \lfloor m/2 \rfloor$. Below in tabular form are avoiding Eulerian circuits together with the number of steps to reach the vertices $u$ and $v$. From this information, it is straightforward to verify that the circuits are always distance at least 2 apart and so are avoiding using similar arguments as employed in the proof of the previous theorem.
		
		\begin{tabular}{rcccccccccc}
			\textnormal{Vertex} $u$ & Circuit $\mathcal{C}_1$ & $u$ & $a$ & $v$ & \textnormal{C} & $u$ & \textnormal{D} & $v$ & \textnormal{B} & $u$\\
			~ & \# steps & 0 & ~ & 2 & ~ & $m+3$ & ~ & $2m+4$ & ~ & $2m+9$\\
			~&&&&&&&&&&\\
			~ & Circuit $\mathcal{C}_2$ & $u$ & \textnormal{B} & $v$ & \textnormal{C} & $u$ & \textnormal{D} & $v$ & $a$ & $u$\\
			~ & \# steps & 0 & ~ & 5 & ~ & $m+6$ & ~ & $2m+7$ & ~ & $2m+9$\\
                ~&&&&&&&&&&\\
		\end{tabular}
  
		\begin{tabular}{ccccccccccc}
			\textnormal{Vertex} $a$ & Circuit $\mathcal{C}_1$ & $a$ & $u$ & $\textnormal{C}$ & $v$ & \textnormal{D} & $u$ & \textnormal{B} & $v$ & $a$\\
			~ & \# steps & 0 & 1 & ~ & $m+2$ & ~ & $2m+3$ & ~ & $2m+8$ & $2m+9$\\
			~&&&&&&&&&&\\
			~ & Circuit $\mathcal{C}_2$ & $a$ & $v$ & \textnormal{B} & $u$ & \textnormal{C} & $v$ & \textnormal{D} & $u$ & $a$\\
			~ & \# steps & 0 & 1 & ~ & 6 & ~ & $m+7$ & ~ & $2m+8$ & $2m+9$\\
                ~&&&&&&&&&&\\
		\end{tabular}
		
		\begin{tabular}{ccccccccccc}
			\textnormal{Vertex} $b_i$ & Circuit $\mathcal{C}_1$ & $b_i$ & $u$ & \textnormal{C} & $v$ & \textnormal{D} & $u$ & $a$ & $v$ & $b_i$\\
			~ & \# steps & 0 & $i$ & ~ & $m+1+i$ & ~ & $2m+2+i$ & ~ & $2m+4+i$ & $2m+9$\\
			~&&&&&&&&&&\\
			~ & Circuit $\mathcal{C}_2$ & $b_i$ & $v$ & $a$ & $u$ & \textnormal{C} & $v$ & \textnormal{D} & $u$ & $b_i$\\
			~ & \# steps & 0 & $5-i$ & ~ & $7-i$ & ~ & $m+8-i$ & ~ & $2m+9-i$ & $2m+9$\\
                ~&&&&&&&&&&\\
		\end{tabular}
		
		~\\For $1 \leq i \leq \lfloor (m-3)/2 \rfloor$,\\
		\begin{tabular}{ccccccccccc}
			\textnormal{Vertex} $c_i$ & Circuit $\mathcal{C}_1$ & $c_i$ & $u$ & \textnormal{D} & $v$ & \textnormal{B} & $u$ & $a$ & $v$ & $c_i$\\
			~ & \# steps & 0 & $i$ & ~ & $m+1+i$ & ~ & $m+6+i$ & ~ & $m+8+i$ & $2m+9$\\
			~&&&&&&&&&&\\
			~ & Circuit $\mathcal{C}_2$ & $c_i$ & $v$ & \textnormal{B} & $u$ & \textnormal{D} & $v$ & $a$ & $u$ & $c_i$\\
			~ & \# steps & 0 & $m+1-i$ & ~ & $m+6-i$ & ~ & $2m+7-i$ & ~ & $2m+9-i$ & $2m+9$\\
                ~&&&&&&&&&&\\
		\end{tabular}
		
		~\\For $\lfloor (m-1)/2 \rfloor \leq i \leq \lfloor m/2 \rfloor$,\\
		\begin{tabular}{ccccccccccc}
			\textnormal{Vertex} $c_i$ & Circuit $\mathcal{C}_1$ & $c_i$ & $u$ & \textnormal{D} & $v$ & \textnormal{B} & $u$ & $a$ & $v$ & $c_i$\\
			~ & \# steps & 0 & $i$ & ~ & $m+1+i$ & ~ & $m+6+i$ & ~ & $m+8+i$ & $2m+9$\\
			~&&&&&&&&&&\\
			~ & Circuit $\mathcal{C}_2$ & $c_i$ & $v$ & $a$ & $u$ & \textnormal{D} & $v$ & \textnormal{B} & $u$ & $c_i$\\
			~ & \# steps & 0 & $m+1-i$ & ~ & $m+3-i$ & ~ & $2m+4-i$ & ~ & $2m+9-i$ & $2m+9$\\
                ~&&&&&&&&&&\\
		\end{tabular}
	\end{proof}
	
	The above theorem gives a construction for edge-minimal doubly Eulerian graphs for all odd orders $n \geq 15$. Again the graphs have edge excess equal to 2. We have already noted that there are no double Eulerian graphs of order 3, 5 or 7. The next case $n=9$ is of particular interest. There are no doubly Eulerian graphs with edge excess equal to 2. Edge-minimal graphs of this order have edge excess equal to 3; there are two such graphs, illustrated in Figure~\ref{fig:n9e12}. For the cases of $n=11$ and $n=13$, the graphs $\Gamma(1,2,3,3)$ and $\Gamma(1,2,4,4)$ respectively are the unique edge-minimal doubly Eulerian graphs of these orders.
	
	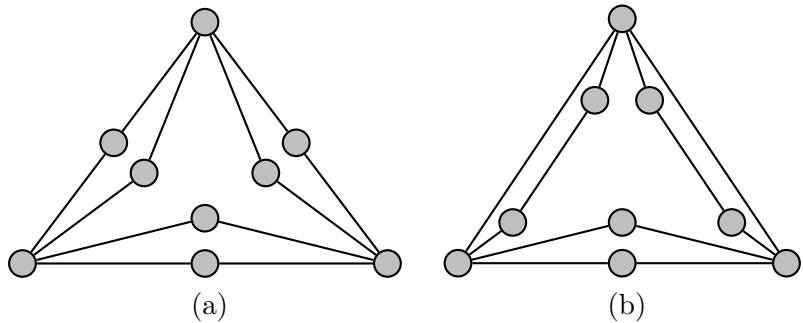
\begin{figure}[h]
		\centering
		\begin{tabular}{cc}
			\begin{tikzpicture}[x=0.2mm,y=-0.2mm,inner sep=0.2mm,scale=1,thick,vertex/.style={circle,draw,minimum size=10,fill=lightgray}]
\node at (320,320) [vertex] (v1) {};
\node at (340,340) [vertex] (v2) {};
\node at (380,400) [vertex] (v3) {};
\node at (380,370) [vertex] (v4) {};
\node at (440,320) [vertex] (v5) {};
\node at (420,340) [vertex] (v6) {};
\node at (260,400) [vertex] (v7) {};
\node at (380,240) [vertex] (v8) {};
\node at (500,400) [vertex] (v9) {};
\path
	(v1) edge (v7)
	(v1) edge (v8)
	(v2) edge (v7)
	(v2) edge (v8)
	(v3) edge (v7)
	(v3) edge (v9)
	(v4) edge (v7)
	(v4) edge (v9)
	(v5) edge (v8)
	(v5) edge (v9)
	(v6) edge (v8)
	(v6) edge (v9)
	;
\end{tikzpicture} &  \begin{tikzpicture}[x=0.2mm,y=-0.2mm,inner sep=0.2mm,scale=0.9,thick,vertex/.style={circle,draw,minimum size=10,fill=lightgray}]
\node at (240,410) [vertex] (v1) {};
\node at (400,410) [vertex] (v2) {};
\node at (320,410) [vertex] (v3) {};
\node at (320,440) [vertex] (v4) {};
\node at (300,320) [vertex] (v5) {};
\node at (340,320) [vertex] (v6) {};
\node at (200,440) [vertex] (v7) {};
\node at (440,440) [vertex] (v8) {};
\node at (320,260) [vertex] (v9) {};
\path
	(v1) edge (v5)
	(v1) edge (v7)
	(v2) edge (v6)
	(v2) edge (v8)
	(v3) edge (v7)
	(v3) edge (v8)
	(v4) edge (v7)
	(v4) edge (v8)
	(v5) edge (v9)
	(v6) edge (v9)
	(v7) edge (v9)
	(v8) edge (v9)
	;
\end{tikzpicture} \\
			(a) & (b)
		\end{tabular}
		\caption{The two doubly Eulerian graphs of order 9 and edge excess 3}
		\label{fig:n9e12}
	\end{figure}
	
	Theorem~\ref{a1b4} is a specific case of the more general theorem below which we give without proof. 
 	\begin{theorem}\label{ba3}
		The graph $\Gamma(a,b,c,d)$ where $a \geq 1$, $b \geq a+3$ and $d < a+b+c+3$ has avoidance index~2.
	\end{theorem}

 This follows similar arguments as employed in the proof of that theorem, but is both lengthy and tedious. This leaves the case where $a \geq 1$ and $a \leq b \leq a+2$, but as stated above we find that this is very complex and divides into a number of different cases and subcases.

	\section{Avoidance index}\label{sec:index}
	We have concentrated so far on graphs which admit a pair of avoiding Eulerian circuits starting at any vertex. A natural extension of this idea is to ask for more than two such circuits at every vertex; as we defined earlier, the maximum number $k$ such that a graph $G$ admits a set of $k$ mutually avoiding Eulerian circuits from any vertex is called the \emph{avoidance index} $\av(G)$. Some simple bounds on the avoidance index are provided by the following lemma.
	
	\begin{lemma}\label{lem:avbounds}
		Let $G$ be an Eulerian graph of order $n$, minimum degree $\delta$ and maximum degree $\Delta$. Then:
		\begin{enumerate}[label=(\roman*),itemsep=0px,topsep=0px]
			\item $\av(G)\leq \delta$;
			\item $\av(G)\leq \displaystyle\min_{v\in V(G)}\alpha(N(v))$, where $\alpha(N(v))$ is the independence number of the subgraph induced by the neighbours of $v$;
			\item If $\Delta=n-1$ then $\av(G)=1$, else	$\av(G)\leq \displaystyle\min_{v\in V(G)}\dfrac{S(v)-2}{\deg(v)}+1$, where $S(v)$ is the sum of the degrees of the non-neighbours of $v$ (other than itself) in $G$. In particular, if $\Delta\leq n-2$ then $\av(G)\leq n-\Delta-1$.
		\end{enumerate}
	\end{lemma}
	\begin{proof}
		Parts (i) and (ii) are immediate since at the first step after the starting vertex, all circuits must be at distinct mutually nonadjacent vertices. The first part of (iii) follows from Lemma~\ref{oddnvaln-1}; note that $n$ must be odd in this case. For the second part, assume $\Delta\leq n-2$; consider a vertex $v$ and suppose we have $k$ mutually avoiding Eulerian circuits beginning and ending at a vertex $u$ which is not adjacent to $v$. Each of these circuits visits $v$ exactly $\deg(v)/2$ times, for a total of $k\deg(v)/2$ such visits across all the circuits. On each of these visits, at the same step the other $k-1$ circuits must all be at some non-neighbour of $v$. In any circuit, the number of visits to a non-neighbour $w$ of $v$, $w\neq u$, is $\deg(w)/2$, and the number of visits to $u$ (excluding the start and end steps) is $\deg(u)/2-1$. So the largest possible number of visits to a non-neighbour of $v$ across all $k$ circuits, excluding at the start and end points, is $k(S(v)/2-1)$. Thus $(k-1)k\deg(v)/2\leq k(S(v)/2-1)$ and the result follows. The final part of (iii) follows by considering a vertex $v$ of degree $\Delta$; in this case $v$ has $n-\Delta-1$ non-neighbours and so $S(v)\leq (n-\Delta-1)\Delta$.
	\end{proof}

	Lemma~\ref{lem:avbounds}(iii) shows that in the case of a $d$-regular Eulerian graph of order $2d$, the maximum possible value of the avoidance index is $d-1$. Our next result shows that this upper bound can be attained for all even $d\geq 2$.

	\begin{theorem}\label{thm:K2s2s}
		Let $G$ be the complete bipartite graph $K_{2s,2s}$, $s\geq 1$. Then $\av(G)=2s-1$.
	\end{theorem}
	
	\begin{proof}
		In the notation of Lemma~\ref{lem:avbounds}, $\delta=\Delta=\alpha=2s$. Therefore $\av(G)\leq 2s-1$. We prove that $\av(G)=2s-1$. The result is trivially true for $s=1$. So assume that $s\geq 2$. Represent the two partitions of the vertices of $G$  as the rows and columns respectively of a square $Q$ of side $2s$, both indexed by the set $\{0,1,2,\ldots,2s-1\}$. The entries in the square then represent the edges of the graph. We number these as follows.
		
		$Q(0,0)=1$, $Q(0,1)=4s$, $Q(1,0)=2$, $Q(1,1)=3$.
		
		$Q(2i,1)=4i$, $Q(2i,2)=4i+1$, $Q(2i+1,2)=4i+2$, $Q(2i+1,1)=4i+3$, $1\leq i\leq s-1$.
		
		$Q(i,j)=Q(i,j-2)+4s$, $0\leq i\leq 1$, $2\leq j\leq 2s-1$ and $2\leq i\leq 2s-1$, $3\leq j\leq 2s$.
		
		Arithmetic in the indices of the square is performed modulo $2s$.
		
		The numbers form a ``zig-zag'' pattern in the square, with each number being alternately in the same column or same row as its predecessor. (The ordering for the case $s=3$ is shown in the following example.) They therefore give an Eulerian circuit beginning and ending at the vertex row 0, with the numbers indicating the order in which the edges are traversed. 	
		This may be seen more easily if the above scheme is described algorithmically as follows, starting with $Q(0,0)=1$. The operations $D$, $R$ and $L$ mean moving down, right and left respectively and inserting the next number, the square $Q$ being regarded cyclically. The algorithm is:
		
		$(D,R,(D,R,D,L)$ performed $(s-1)$ times, $D,R)$ performed $s$ times.
		
		Any numbering of the square $Q$ which has the ``zig-zag'' pattern as described above will represent an Eulerian circuit. In order to construct mutually avoiding Eulerian circuits, all beginning and ending at the vertex row 0, the numbers $1$ and $4s^2$ must remain in row 0 but be in different columns in each square. Every other number must be in a different row and different column in each square.
		
		This is achieved by defining squares $E_i$ and $E'_i$, $1\leq i\leq 2s-1$, as follows. First let $E_1=E'_1=Q$. Construct square $E'_i$ from square $E'_{i-1}$, $2\leq i\leq 2s-1$, by moving row 0 one column to the right (cyclically) and every other row one column to the right (cyclically) and one column upwards, with row 1 becoming row $2s-1$. Finally, construct square $E_i$ from square $E'_i$, $2\leq i\leq 2s-1$, by interchanging the entries in row 0, apart from $1$ and $4s^2$, of $E'_i$ with the entries in the same column of row $i$. The squares $E_i$, $1\leq i\leq 2s-1$, give a set of $2s-1$ mutually avoiding Eulerian circuits of the graph $K_{2s,2s}$.
	\end{proof}
	
	The construction is illustrated by the following example for the case $s=3$.
	
	\begin{example}
		Five mutually avoiding Eulerian circuits for the complete bipartite graph $K_{6,6}$ beginning and ending at the same vertex, and presented as squares as defined in Theorem~\ref{thm:K2s2s}.
		
		\centering
		\setlength{\tabcolsep}{3pt}
		\begin{tabular}{c|cccccc}
			& 0 & 1 & 2 & 3 & 4 & 5 \\
			\hline
			0 & 1 & 12 & 13 & 24 & 25 & 36 \\
			1 & 2 & 3 & 14 & 15 & 26 & 27 \\
			2 & 29 & 4 & 5 & 16 & 17 & 28 \\
			3 & 30 & 7 & 6 & 19 & 18 & 31 \\
			4 & 33 & 8 & 9 & 20 & 21 & 32 \\
			5 & 34 & 11 & 10 & 23 & 22 & 35 \\
		\end{tabular}
		\quad
		\begin{tabular}{c|cccccc}
			& 0 & 1 & 2 & 3 & 4 & 5 \\
			\hline
			0 & 36 & 1 & 7 & 6 & 19 & 18 \\
			1 & 28 & 29 & 4 & 5 & 16 & 17 \\
			2 & 31 & 30 & 12 & 13 & 24 & 25 \\
			3 & 32 & 33 & 8 & 9 & 20 & 21 \\
			4 & 35 & 34 & 11 & 10 & 23 & 22 \\
			5 & 27 & 2 & 3 & 14 & 15 & 26 \\
		\end{tabular}
		\quad
		\begin{tabular}{c|cccccc}
			& 0 & 1 & 2 & 3 & 4 & 5 \\
			\hline
			0 & 22 & 36 & 1 & 11 & 10 & 23 \\
			1 & 18 & 31 & 30 & 7 & 6 & 19 \\
			2 & 21 & 32 & 33 & 8 & 9 & 20 \\
			3 & 25 & 35 & 34 & 12 & 13 & 24 \\
			4 & 26 & 27 & 2 & 3 & 14 & 15 \\
			5 & 17 & 28 & 29 & 4 & 5 & 16 \\
		\end{tabular}
		
		\begin{tabular}{c|cccccc}
			& 0 & 1 & 2 & 3 & 4 & 5 \\
			\hline
			0 & 16 & 17 & 36 & 1 & 4 & 5 \\
			1 & 20 & 21 & 32 & 33 & 8 & 9 \\
			2 & 23 & 22 & 35 & 34 & 11 & 10 \\
			3 & 15 & 26 & 27 & 2 & 3 & 14 \\
			4 & 24 & 25 & 28 & 29 & 12 & 13 \\
			5 & 19 & 18 & 31 & 30 & 7 & 6 \\
		\end{tabular}
		\quad
		\begin{tabular}{c|cccccc}
			& 0 & 1 & 2 & 3 & 4 & 5 \\
			\hline
			0 & 9 & 20 & 21 & 36 & 1 & 8 \\
			1 & 10 & 23 & 22 & 35 & 34 & 11 \\
			2 & 14 & 15 & 26 & 27 & 2 & 3 \\
			3 & 5 & 16 & 17 & 28 & 29 & 4 \\
			4 & 6 & 19 & 18 & 31 & 30 & 7 \\
			5 & 13 & 24 & 25 & 32 & 33 & 12 \\
		\end{tabular}
	\end{example}

	The above theorem can easily be extended to complete multipartite graphs. 
	
	\begin{corollary}\label{cor:multipartite}
	Let $G$ be the complete multipartite graph $K_{2s,2s,\ldots,2s}$ of order $n=2ks$, $s\geq 2$, $k\geq 2$. Then $\av(G)=2s-1$.
	\end{corollary}
	\begin{proof}
	From Lemma~\ref{lem:avbounds}, $\av(G)\leq n-\Delta-1=2s-1$. The proof now proceeds by induction. It is true for $k=2$ by Theorem~\ref{thm:K2s2s}. Let $k\geq 3$ and assume that the result is true for the graph $K_{2s,2s,\ldots,2s}$ of order $2(k-1)s$. Now apply the procedure in Proof 2 of Theorem~\ref{thm:max0mod4} to the $2s-1$ avoiding Eulerian circuits. We then have $2s-1$ avoiding circuits of the graph $K_{2s,2s,\ldots,2s}$ of order $2ks$. Thus $\av(G)=2s-1$.
	\end{proof}
	
	We now briefly turn our attention to complete bipartite graphs $K_{2r,2s}$, $1\leq r<s$. When $r=1$, such graphs have avoidance index 1, as can readily be seen by an Eulerian circuit starting at either point of the 2-partition. However, it is easy to find two mutually avoiding Eulerian circuits beginning and ending at any vertex of the larger partition. This simple observation prompts the following definitions.
	
	The \emph{avoidance index} $\av(v)$ of a vertex $v\in V(G)$ of an Eulerian graph $G$ is the maximum number of mutually avoiding Eulerian circuits beginning and ending at $v$. This is related to the avoidance index of the graph in the obvious way:
	\[\av(G)=\min_{v\in V(G)}\av(v).\]
	We may now define the \emph{mean avoidance index} $\eav(G)$ to be the average avoidance index across all vertices:
	\[\eav(G)=\frac{\displaystyle\sum_{v\in V(G)}\av(v)}{\lvert V(G)\rvert}.\]
	Thus for the complete bipartite graph $K_{2,2s}$,  $\eav(K_{2,2s})=(4s+2)/(2s+2)=2-1/(s+1)$, which asymptotically tends to 2 as $s$ tends to infinity.
	
	It is not the intention in this paper to investigate a number of interesting questions which arise with the mean avoidance index, but the observations immediately above can easily be extended to graphs $K_{2r,2s}$, $r<s$, where $r\geq 2$. Here $\delta=2r$ and $\Delta=\alpha=2s$. Therefore $\av(K_{2r,2s})\leq 2r-1$. The number of mutually avoiding Eulerian circuits beginning and ending at any vertex can be determined either by constructing a $2r\times 2s$ rectangle or a $2s\times 2r$ rectangle, and numbering as described in Theorem~\ref{thm:K2s2s}. From a vertex $v$ in the $2r$-partition we have $\av(v)=2r-1$, and in the $2s$-partition $\av(v)=2r$. We therefore have the following result.
	
	\begin{theorem}\label{thm:K2r2s}
		Let $G$ be the complete bipartite graph $K_{2r,2s}$, $1\leq r<s$. Then $\av(G)=2r-1$ and $\eav(G)=2r-r/(r+s)$.
	\end{theorem}
	
	\begin{example}
		In order to illustrate the above argument, relevant rectangles for the complete bipartite graph $K_{4,6}$ are given below.
		
		\centering
		\setlength{\tabcolsep}{3pt}
		\begin{tabular}{c|cccccc}
			& 0 & 1 & 2 & 3 & 4 & 5 \\
			\hline
			0 & 1 & 8 & 9 & 16 & 17 & 24 \\
			1 & 2 & 3 & 10 & 11 & 18 & 19 \\
			2 & 21 & 4 & 5 & 12 & 13 & 20 \\
			3 & 22 & 7 & 6 & 15 & 14 & 23 \\
		\end{tabular}
		\quad
		\begin{tabular}{c|cccccc}
			& 0 & 1 & 2 & 3 & 4 & 5 \\
			\hline
			0 & 24 & 1 & 7 & 6 & 15 & 14 \\
			1 & 20 & 21 & 4 & 5 & 12 & 13 \\
			2 & 23 & 22 & 8 & 9 & 16 & 17 \\
			3 & 19 & 2 & 3 & 10 & 11 & 18 \\
		\end{tabular}
		\quad
		\begin{tabular}{c|cccccc}
			& 0 & 1 & 2 & 3 & 4 & 5 \\
			\hline
			0 & 13 & 24 & 1 & 4 & 5 & 12 \\
			1 & 14 & 23 & 22 & 7 & 6 & 15 \\
			2 & 18 & 19 & 2 & 3 & 10 & 11 \\
			3 & 17 & 20 & 21 & 8 & 9 & 16 \\
		\end{tabular}
		
		\begin{tabular}{c|cccc}
			& 0 & 1 & 2 & 3 \\
			\hline
			0 & 1 & 12 & 13 & 24 \\
			1 & 2 & 3 & 14 & 15 \\
			2 & 17 & 4 & 5 & 16 \\
			3 & 18 & 7 & 6 & 19 \\
			4 & 21 & 8 & 9 & 20 \\
			5 & 22 & 11 & 10 & 23 \\
		\end{tabular}
		\quad
		\begin{tabular}{c|cccc}
			& 0 & 1 & 2 & 3 \\
			\hline
			0 & 24 & 1 & 7 & 6 \\
			1 & 16 & 17 & 4 & 5 \\
			2 & 19 & 18 & 12 & 13 \\
			3 & 20 & 21 & 8 & 9 \\
			4 & 23 & 22 & 11 & 10 \\
			5 & 15 & 2 & 3 & 14 \\
		\end{tabular}
		\quad
		\begin{tabular}{c|cccc}
			& 0 & 1 & 2 & 3 \\
			\hline
			0 & 10 & 24 & 1 & 11 \\
			1 & 6 & 19 & 18 & 7 \\
			2 & 9 & 20 & 21 & 8 \\
			3 & 13 & 23 & 22 & 12 \\
			4 & 14 & 15 & 2 & 3 \\
			5 & 5 & 16 & 17 & 4 \\
		\end{tabular}
		\quad
		\begin{tabular}{c|cccc}
			& 0 & 1 & 2 & 3 \\
			\hline
			0 & 4 & 5 & 24 & 1 \\
			1 & 8 & 9 & 20 & 21 \\
			2 & 11 & 10 & 23 & 22 \\
			3 & 3 & 14 & 15 & 2 \\
			4 & 12 & 13 & 16 & 17 \\
			5 & 7 & 6 & 19 & 18 \\
		\end{tabular}	
	\end{example}
	
	Next, it is natural to consider the graph $K^*_{2r+1,2r+1}$, $r\geq 1$, the complete bipartite graph minus a perfect matching. From Theorem~\ref{thm:bipodd}, for $r\geq 2$ $\av(K^*_{2r+1,2r+1})\geq 2$ and from Lemma~\ref{lem:avbounds}, $\av(K^*_{2r+1,2r+1})\leq 2r$.
	
	When $r=2$, $\av(K^*_{5,5})=4$. Consider the complete bipartite graph $K_{5,5}$ where the vertices of the two partitions are $\{0,2,4,6,8\}$ and $\{1,3,5,7,9\}$ respectively. Remove the edges $0\edge 5$, $1\edge 6$, $2\edge 7$, $3\edge 8$ and $4\edge 9$ to give $K^*_{5,5}$. Four mutually avoiding Eulerian circuits both starting and ending at the vertex 0 are the following.

	Circuit $C_1$: 0\ 1\ 4\ 3\ 0\ 7\ 8\ 9\ 6\ 5\ 8\ 1\ 2\ 5\ 4\ 7\ 6\ 3\ 2\ 9\ 0
	
	Circuit $C_2$: 0\ 9\ 8\ 7\ 6\ 9\ 2\ 1\ 0\ 3\ 2\ 5\ 4\ 3\ 6\ 5\ 8\ 1\ 4\ 7\ 0

	Circuit $C_3$: 0\ 3\ 2\ 1\ 4\ 5\ 6\ 3\ 4\ 7\ 6\ 9\ 0\ 7\ 8\ 9\ 2\ 5\ 8\ 1\ 0

	Circuit $C_4$: 0\ 7\ 6\ 5\ 2\ 3\ 4\ 5\ 8\ 1\ 4\ 7\ 8\ 9\ 2\ 1\ 0\ 9\ 6\ 3\ 0
	
	These circuits were found by computer. We also have an example of six mutually avoiding Eulerian circuits of $K^*_{7,7}$. Again we begin with the complete bipartite graph $K_{7,7}$ where the vertices of the two partitions are $\{0,2,4,6,8,10,12\}$ and $\{1,3,5,7,9,11,13\}$ respectively. We remove the edges $i\edge i+7$, $0 \leq i \leq 6$; then six mutually avoiding circuits $C_1,\ldots,C_6$ in the resulting $K^*_{7,7}$ are as follows.

 \newpage
	0\ 1\ 2\ 3\ 4\ 7\ 8\ 9\ 10\ 13\ 0\ 3\ 6\ 9\ 12\ 11\ 10\ 7\ 6\ 5\ 4\ 9\ 0\ 11\ 2\ 13\ 8\ 3\ 12\ 13\ 4\ 1\ 6\ 11\ 8\ 5\ 2\ 7\ 12\ 1\ 10\ 5\ 0
	
	0\ 13\ 12\ 7\ 8\ 9\ 10\ 7\ 6\ 5\ 4\ 7\ 2\ 3\ 4\ 1\ 2\ 5\ 0\ 1\ 12\ 3\ 8\ 13\ 10\ 1\ 6\ 11\ 8\ 5\ 10\ 11\ 12\ 9\ 0\ 3\ 6\ 9\ 4\ 13\ 2\ 11\ 0
	
	0\ 3\ 4\ 1\ 2\ 3\ 6\ 5\ 4\ 9\ 10\ 1\ 12\ 13\ 2\ 7\ 6\ 11\ 10\ 13\ 0\ 1\ 6\ 9\ 12\ 11\ 0\ 5\ 10\ 7\ 12\ 3\ 8\ 7\ 4\ 13\ 8\ 5\ 2\ 11\ 8\ 9\ 0
	
	0\ 11\ 10\ 5\ 6\ 11\ 12\ 13\ 2\ 3\ 6\ 9\ 10\ 7\ 8\ 9\ 12\ 1\ 2\ 11\ 8\ 13\ 10\ 1\ 0\ 5\ 2\ 7\ 4\ 3\ 0\ 9\ 4\ 1\ 6\ 7\ 12\ 3\ 8\ 5\ 4\ 13\ 0
	
	0\ 5\ 6\ 9\ 10\ 13\ 2\ 1\ 12\ 11\ 8\ 5\ 4\ 1\ 6\ 3\ 8\ 13\ 4\ 7\ 6\ 11\ 2\ 3\ 4\ 9\ 12\ 13\ 0\ 9\ 8\ 7\ 2\ 5\ 10\ 11\ 0\ 1\ 10\ 7\ 12\ 3\ 0
	
	0\ 9\ 8\ 11\ 12\ 1\ 4\ 3\ 8\ 7\ 12\ 13\ 8\ 5\ 10\ 13\ 4\ 9\ 12\ 3\ 2\ 7\ 4\ 5\ 6\ 7\ 10\ 9\ 6\ 11\ 2\ 5\ 0\ 13\ 2\ 1\ 10\ 11\ 0\ 3\ 6\ 1\ 0
	
	On the basis of this, admittedly rather slim, evidence we conjecture that $\av(K^*_{2r+1,2r+1})=2r$, $r\geq 2$.
	
	The following definition is prompted by the above theorem and examples. An Eulerian graph is said to be \emph{saturated} if it is $d$-regular and admits a full set of $d$ mutually avoiding Eulerian circuits. Clearly $d$ cannot be 2, so the first case to consider is when $d=4$. Table~\ref{tab:ind4reg} shows the results of computer calculations for 4-regular graphs of small orders.
	
	\begin{table}[h]
		\centering
		\begin{tabular}{|c|c|cccc|}
			\hline
			Order & 4-regular & \multicolumn{4}{|c|}{Avoidance index} \\
			& Eulerian graphs & 1 & 2 & 3 & 4 \\
			\hline
			5 & 1 & 1 & 0 & 0 & 0 \\
			6 & 1 & 1 & 0 & 0 & 0 \\
			7 & 2 & 2 & 0 & 0 & 0 \\
			8 & 6 & 1 & 4 & 1 & 0 \\
			9 & 16 & 2 & 14 & 0 & 0 \\
			10 & 59 & 2 & 55 & 1 & 1 \\
			11 & 265 & 3 & 242 & 20 & 0 \\
			12 & 1544 & 1 & 1225 & 314 & 4 \\
			13 & 10778 & 3 & 8194 & 2576 & 5 \\
			\hline
		\end{tabular}
		\caption{Avoidance index of 4-regular graphs}
		\label{tab:ind4reg}
	\end{table}
	
	Thus the smallest example of a 4-regular saturated Eulerian graph is $K^*_{5,5}$ of order 10. The graph of order 10 and avoidance index 3 is a circulant; on the group $\Z_{10}$ it has generating set $\{2,3\}$. Three of the saturated graphs of order 12 are Cayley graphs. Two are also circulants; on the group $\Z_{12}$ they have generating sets $\{1,3\}$ and $\{1,5\}$ respectively. The third graph is on the group $\Z_6\times\Z_2$ and has generating set $\{(1,0),(3,0),(0,1)\}$. The fourth graph is not a Cayley graph, and indeed it is not vertex-transitive. It is the graph $K_{6,6}$ from which an 8-cycle and a 4-cycle have been removed, and it is shown in Figure~\ref{fig:n12ind4}. 
	
	\begin{figure}[h]
		\centering
		\begin{tikzpicture}[x=0.2mm,y=0.2mm,very thick,vertex/.style={circle,draw,minimum size=10,inner sep=0,fill=lightgray}]
	\node at (-210,0) [vertex] (v1) {};
	\node at (-90,0) [vertex] (v2) {};
	\node at (150,120) [vertex] (v3) {};
	\node at (150,60) [vertex] (v4) {};
	\node at (150,-60) [vertex] (v5) {};
	\node at (150,-120) [vertex] (v6) {};
	\node at (-150,120) [vertex] (v7) {};
	\node at (-150,-120) [vertex] (v8) {};
	\node at (-150,60) [vertex] (v9) {};
	\node at (-150,-60) [vertex] (v10) {};
	\node at (90,0) [vertex] (v11) {};
	\node at (210,0) [vertex] (v12) {};
	\draw (v1) to (v7);
	\draw (v1) to (v8);
	\draw (v1) to (v9);
	\draw (v1) to (v10);
	\draw (v2) to (v7);
	\draw (v2) to (v8);
	\draw (v2) to (v9);
	\draw (v2) to (v10);
	\draw (v3) to (v7);
	\draw (v3) to (v9);
	\draw (v3) to (v11);
	\draw (v3) to (v12);
	\draw (v4) to (v7);
	\draw (v4) to (v10);
	\draw (v4) to (v11);
	\draw (v4) to (v12);
	\draw (v5) to (v8);
	\draw (v5) to (v9);
	\draw (v5) to (v11);
	\draw (v5) to (v12);
	\draw (v6) to (v8);
	\draw (v6) to (v10);
	\draw (v6) to (v11);
	\draw (v6) to (v12);
\end{tikzpicture}
		\caption{A saturated 4-regular graph of order 12}
		\label{fig:n12ind4}
	\end{figure}
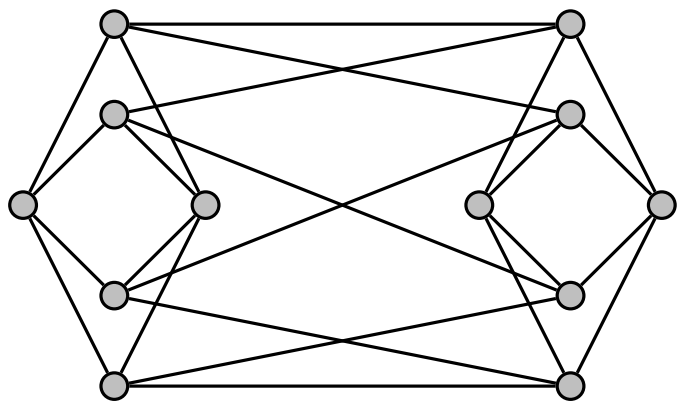
	
	All of these graphs are bipartite. The smallest example of a non-bipartite saturated 4-regular Eulerian graph occurs at order 13: the circulant graph on the group $\Z_{13}$ with generating set $\{1,3\}$. It would be good to find an infinite class of saturated 4-regular Eulerian graphs. Prime candidates are circulant graphs on the group $\Z_n$ with various generating sets. For generating set $\{1,3\}$ the graphs are saturated for $n=10$ or $12\leq n\leq 30$, for generating set $\{2,3\}$ they are saturated for $14\leq n\leq 30$ and for generating set $\{1,4\}$ they are saturated for $13\leq n\leq 30$. Again these results were found by computer.
	
	It may be of some interest that the vast majority of the graphs considered in Table~\ref{tab:ind4reg} are doubly Eulerian; indeed for order 12 only one is not. This exceptional graph is illustrated in Figure~\ref{fig:n12ind1}. Clearly this graph has vertex connectivity 1; this is no coincidence as our next result shows.
	
	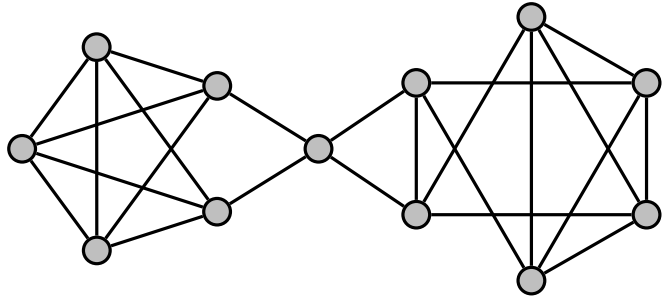
\begin{figure}
		\centering
		\begin{tikzpicture}[x=0.2mm,y=0.2mm,very thick,vertex/.style={circle,draw,minimum size=10,inner sep=0,fill=lightgray}]
	\node at (-164.2,67.4) [vertex] (v1) {};
	\node at (197.4,43.7) [vertex] (v2) {};
	\node at (46,-43.7) [vertex] (v3) {};
	\node at (-213.2,0) [vertex] (v4) {};
	\node at (121.7,-87.4) [vertex] (v5) {};
	\node at (-18.3,0) [vertex] (v6) {};
	\node at (-164.1,-67.4) [vertex] (v7) {};
	\node at (121.7,87.4) [vertex] (v8) {};
	\node at (-84.9,41.7) [vertex] (v9) {};
	\node at (-85,-41.7) [vertex] (v10) {};
	\node at (46,43.7) [vertex] (v11) {};
	\node at (197.4,-43.7) [vertex] (v12) {};
	\draw (v1) to (v4);
	\draw (v1) to (v7);
	\draw (v1) to (v9);
	\draw (v1) to (v10);
	\draw (v2) to (v5);
	\draw (v2) to (v8);
	\draw (v2) to (v11);
	\draw (v2) to (v12);
	\draw (v3) to (v6);
	\draw (v3) to (v8);
	\draw (v3) to (v11);
	\draw (v3) to (v12);
	\draw (v4) to (v7);
	\draw (v4) to (v9);
	\draw (v4) to (v10);
	\draw (v5) to (v8);
	\draw (v5) to (v11);
	\draw (v5) to (v12);
	\draw (v6) to (v9);
	\draw (v6) to (v10);
	\draw (v6) to (v11);
	\draw (v7) to (v9);
	\draw (v7) to (v10);
	\draw (v8) to (v12);
\end{tikzpicture}
		\caption{The unique 4-regular graph of order 12 which is not doubly Eulerian}
		\label{fig:n12ind1}
	\end{figure}

	\begin{theorem}\label{thm:4regind1}
	For every $n\geq 5$, there exists a 4-regular graph of order $n$ which is not doubly Eulerian.
	\end{theorem}
	\begin{proof}
	The result is true for $5\leq n\leq 10$ by Table~\ref{tab:ind4reg}. Let $n\geq 11$ and suppose first that $n=2s+1$ is odd. Let $H$ be any 4-regular graph of order $s$ and delete one edge from $H$. Create a graph $G$ by taking two disjoint copies of $H$ and adding a new vertex $v$ which is joined to both vertices of degree 3 in both copies of $H$. Any Eulerian circuit starting and ending at $v$ must traverse all the edges in one of the copies of $H$ before returning to $v$, since $v$ is a cut vertex and is revisited only once on the circuit. Thus any two Eulerian circuits starting and ending at $v$ must collide at this point.
	
	Now suppose $n=2s$ is even, $s\geq 6$. Choose any 4-regular graph $H$ of order $s-1$ and delete one edge from it. Let $K$ be a graph of order $s$ which has $s-2$ vertices of order 4 and exactly two vertices of order 3, which must be adjacent. (Such a graph can always be constructed: we start with a circulant graph of order $s$ and connection set $\{-2,-1,1,2\}$, delete the edges from $0$ to $-1$ and $1$ to $2$ and add a new edge from $-1$ to $2$.) As before, construct a graph $G$ by taking a copy of $H$ and $K$ plus a new vertex $v$ which is joined to the degree 3 vertices in $H$ and $K$. Suppose we have a pair of mutially avoiding Eulerian circuits starting and ending at $v$. By the argument in the case of odd $n$, one circuit must visit all the edges of $H$ before returning to $v$; the other must have been in $K$ and at the point when the first circuit is back at $v$, the second is two steps away from $v$. After one more step, the circuits are on the pair of vertices which have degree 3 in $K$, and these are adjacent. (This situation is clearly seen by referring to Figure~\ref{fig:n12ind1}.)
	\end{proof}

	It is interesting to note that there are exactly three ways to construct a graph $G$ of order 13 in the manner of the proof of Theorem~\ref{thm:4regind1}, and from Table~\ref{tab:ind4reg} these account for all the 4-regular graphs of order 13 which are not doubly Eulerian.
	
	Turning now to 6-regular graphs, by Lemma~\ref{lem:avbounds} the smallest possible order of a doubly Eulerian graph is 9. There are four Eulerian graphs of order 9, each of which is a complete graph $K_9$ from which a 2-factor has been removed. The two which have respectively either a 9-cycle or three 3-cycles removed are circulants, and therefore vertex-transitive. These both have avoidance index 2. The other two which have respectively either a 6-cycle and a 3-cycle or a 5-cycle and a 4-cycle removed only have avoidance index 1.
	
	There are 21 6-regular graphs of order 10. By Lemma~\ref{lem:avbounds}, the maximum possible avoidance index is 3. In fact, no 6-regular graph of order 10 has avoidance index 3. The computations required to determine the split of these 21 graphs between avoidance index 1 and 2 are substantial, and we have not attempted this.
	
	Finally in this section we present a doubling construction for Eulerian circulant graphs.
	\begin{theorem}\label{thm:gendoub}
	Let $G$ be a an Eulerian circulant graph of order $n$ and degree $d$. Suppose that $G$ admits a set of $k$ mutually avoiding Eulerian circuits. Then there exists a graph of order $2n$ and degree $2d$ which also admits $k$ mutually avoiding Eulerian circuits.
	\end{theorem}
	\begin{proof}
	We follow an amended version of Proof 2 of Theorem~\ref{thm:max2mod4}. Let the vertices of $G$ be $\{0,1,\ldots,n-1\}$ and suppose the connection set of $G$ is $S=\{s_1,s_2,\ldots,s_d\}$. Create a new graph $G'$ by adding new vertices $\{0',1',\ldots,(n-1)'\}$. In $G$ the existing edges are of the form $i\edge(i+s)$ for every $s\in S$; in $G'$ we add the edges $i\edge(i+s)'$ and $i'\edge(i+s)'$ for every $s\in S$ (with arithmetic modulo $n$). Let $H'$ be the subgraph of $G'$ consisting of all the edges of the forms $i\edge(i+s)'$ and $i'\edge(i+s)'$. Then $H'$ is not regular, but has degrees $d$ and $2d$ and hence is Eulerian.
	
	Consider a set $\{C_1,C_2,\ldots,C_k\}$ of mutually avoiding Eulerian circuits in $G$ starting and ending at vertex 0. Construct an Eulerian circuit $C'_1$ in $G'$ as follows. Follow $C_1$ until the first return to vertex 0 after the start point; then follow any Eulerian circuit in $H'$ starting and ending at 0; then follow the remainder of $C_1$. Now construct a circuit $C'_2$ by following $C_2$ until the point that $C_1$ has returned to 0; at this point $C_2$ will be at some vertex $z$. Now continue $C'_2$ by following a `parallel' circuit to that followed by $C'_1$, i.e. if $C'_1$ visits vertex $x$ or $x'$ then $C'_2$ visits $x+z$ or $(x+z)'$ respectively (arithmetic modulo $n$). The parallel circuit ends when all the edges of $H'$ have been visited, so $C'_2$ is back at vertex $z$ (and $C'_1$ is back at 0). Now continue with the remaining edges of $C_2$. Continue in a similar way to construct circuits $C'_3,\ldots,C'_k$. The set $\{C'_1,C'_2,\ldots,C'_k\}$ is a set of $k$ mutually avoiding Eulerian circuits in $G'$.
	\end{proof}
	
	\section{Summary and future research}\label{sec:summary}
	In this paper we have introduced the concepts of a doubly Eulerian graph and more generally, the avoidance index of an Eulerian graph. Not all Eulerian graphs are doubly Eulerian; indeed as we observed in the Introduction the extremal ones, i.e. the complete graphs on an odd number of vertices and the cycles are not. It was natural therefore to study the question of what could be said about extremal doubly Eulerian graphs and construct examples.
	
	In Section~\ref{sec:maximal}, we proved that for a graph of odd order $n$, an edge-maximal doubly Eulerian graph must be a complete graph minus a 2-factor. We proved in Theorem~\ref{thm:oddncirc} that there exist infinite classes with a complete cycle removed; and in Theorem~\ref{thm:3mod6}, when $n$ is divisible by 3, a set of $n/3$ disjoint triangles removed. From Lemma~\ref{lem:avbounds}, the avoidance index of regular graphs of odd order $n$ and degree $n-3$ is either 1 or 2. When $n=9$, removing cycles of the same (resp. different) lengths gives a graph with avoidance index 2 (resp. 1). Is this the case more generally? The situation for $n=11$, $13$ and $15$ may repay further investigation. We present this as the first problem.
	\begin{problem}
		Classify which regular graphs of odd order $n$ and degree $n-3$ have avoidance index either 1 or 2.
	\end{problem}
	For graphs of even order $n$, an edge-maximal doubly Eulerian graph must be a complete graph from which a regular cubic graph, not necessarily connected, has been removed. Again from Lemma~\ref{lem:avbounds}, the avoidance index of regular graphs of even order $n$ and degree $n-4$ is 1, 2 or 3. From Corollary~\ref{cor:multipartite}, when $n$ is doubly even removing a set of $n/4$ disjoint complete graphs $K_4$ gives a graph with avoidance index 3. The six regular graphs of order 8 and degree 4 were discussed in Section~\ref{sec:maximal}. Of these, $K_{4,4}$ has avoidance index 3, four graphs have avoidance index 2 and there is a unique graph with avoidance index 1. We have not attempted an investigation of the 8-regular graphs of order 12, of which there are 94, but for someone with access to powerful computer facilities we give this as the second problem.
	\begin{problem}
		Determine the avoidance index of the 94 8-regular graphs of order 12.
	\end{problem}
	Turning now to complete bipartite graphs, we have proved (Theorems~\ref{thm:K2s2s} and~\ref{thm:K2r2s}) that the avoidance index of the graph $K_{2s,2s}$, $s\geq 1$, is $2s-1$ and of the graph $K_{2r,2s}$, $1\leq r<s$, is $2r-1$. For the graph $K^*_{2r+1,2r+1}$, which is the complete bipartite graph $K_{2r+1,2r+1}$ minus a perfect matching, we have shown that for $r=2,3$ the avoidance index is $2r$; we conjecture that this holds for all $r\geq 2$. The third problem is to prove this.
	\begin{problem}\label{prob:Kstar}
		Prove that $\av(K^*_{2r+1,2r+1})=2r$, $r\geq 2$.
	\end{problem}
	If the conjecture of Problem~\ref{prob:Kstar} is true then the graphs $K^*_{2r+1,2r+1}$, $r\geq 2$ are examples of saturated graphs, i.e. those whose avoidance index is equal to the degree of the regular graph. As we observed in Section~\ref{sec:index}, there can be no saturated graphs of degree 2 and we have examined the situation for degree 4 for graphs of order 13 or less. But we know of no infinite class. Based on our computer calculations outlined in Section~\ref{sec:index}, we conjecture that circulant graphs with various generating sets may provide such classes and this is the next problem. The conjecture also applies to regular graphs of degree 6 (and indeed of higher degree). The circulant graphs on $\Z_{14}$, $\Z_{16}$ and $\Z_{18}$ with generating set $\{1,3,5\}$ are also saturated. The first of these is of course the graph $K^*_{7,7}$.
	\begin{problem}
		Prove that the circulant graphs on the group $\Z_n$ with various generating sets are saturated for large enough $n$.
	\end{problem}
	The smallest saturated non-Cayley graph is the graph $K_{6,6}$ from which an 8-cycle and a 4-cycle have been removed. It is tempting to speculate that this may be the smallest example of an infinite class, but notwithstanding this leads to the next investigation.
	\begin{problem}
		Find more examples of non-Cayley regular graphs which are saturated.
	\end{problem}
	With the exception of the case where $n=9$, edge-minimal doubly Eulerian graphs have edge excess equal to 2. The structure of such graphs is that they have two vertices of degree 4 connected by four paths containing $a$, $b$, $c$ and $d$ valency 2 vertices respectively, $0\leq a\leq b\leq c\leq d$, and are thus characterised by these four parameters. In the important case where the two valency 4 vertices are connected, i.e. $a=0$, Lemma~\ref{abc3} and Theorem~\ref{a0} provide a complete analysis of which graphs have avoidance index 1 and which have avoidance index 2. This enabled us to find edge-minimal doubly Eulerian graphs for all even orders $n\geq 8$. The graphs are bipartite. Our investigations suggest that a complete analysis of which edge excess 2 graphs have avoidance index 1 or 2 is extraordinarily long and tedious. Non-bipartite edge-minimal doubly Eulerian graphs of even order and with edge excess equal to 2 do exist; the smallest is of order 14. This is the next problem.
	\begin{problem}
		Prove that there exists a non-bipartite edge-minimal doubly Eulerian graph of order $n$ with edge excess equal to 2 for all even $n\geq 14$.
	\end{problem}
	We also proved that there exist (edge-minimal) doubly Eulerian graphs with edge excess equal to 2 for all odd $n\geq 11$. These are not bipartite; indeed it is easy to see that no bipartite graph of odd order and with edge excess equal to 2 can exist. Either one of the parameters $a,b,c,d$ is odd and the other three are even or vice-versa. Thus edge-minimal bipartite graphs of odd order must have edge excess at least 3. The two examples for $n=9$ are shown in Figure~\ref{fig:n9e12}.
	\begin{problem}
		Prove that there exists a bipartite edge-minimal doubly Eulerian graph of order $n$ with edge excess equal to 3 for all odd $n\geq 9$.
	\end{problem}
	This leads naturally to the next investigation. An Eulerian graph with edge excess equal to 2 cannot be Hamiltonian. So if we wish to find edge-minimal Hamiltonian doubly Eulerian graphs, they must have edge excess at least 3. The structure of these graphs with edge excess 3 is immediate; such a graph on $n$ vertices must consist of an $n$-cycle (the Hamiltonian cycle), with $n-3$ vertices of valency 2 and the other three vertices, say $u,v,w$ of valency 4 forming a triangle. Let the number of vertices along the Hamiltonian cycle between vertices $u$ and $v$, $v$ and $w$ and $w$ and $u$ be $a,b,c$ respectively, where $a+b+c=n-3$. Without loss of generality we may assume that $a\leq b\leq c$. Denote this graph by $\Lambda=\Lambda(V,E)=\Lambda(a,b,c)$. An example is shown in Figure~\ref{fig:Lambda457}.
	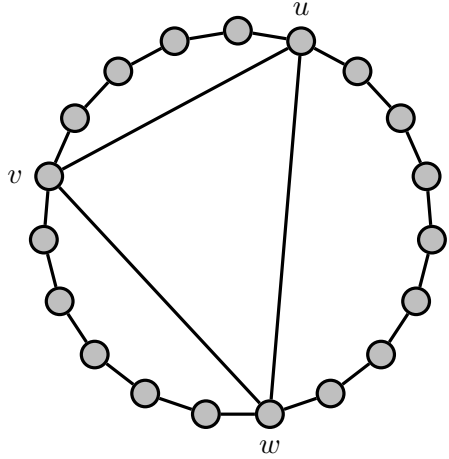
\begin{figure}\centering
		\begin{tikzpicture}[x=0.2mm,y=0.2mm,scale=0.8,very thick,vertex/.style={circle,draw,minimum size=10,inner sep=0,fill=lightgray}]
	\node at (0,160) [vertex] (v1) {};
	\node at (52,151.3) [vertex,label=above:{$u$}] (v2) {};
	\node at (98.3,126.3) [vertex] (v3) {};
	\node at (133.9,87.5) [vertex] (v4) {};
	\node at (155.1,39.3) [vertex] (v5) {};
	\node at (159.5,-13.2) [vertex] (v6) {};
	\node at (146.5,-64.3) [vertex] (v7) {};
	\node at (117.7,-108.4) [vertex] (v8) {};
	\node at (76.2,-140.7) [vertex] (v9) {};
	\node at (26.3,-157.8) [vertex,label=below:{$w$}] (v10) {};
	\node at (-26.3,-157.8) [vertex] (v11) {};
	\node at (-76.2,-140.7) [vertex] (v12) {};
	\node at (-117.7,-108.4) [vertex] (v13) {};
	\node at (-146.5,-64.3) [vertex] (v14) {};
	\node at (-159.5,-13.2) [vertex] (v15) {};
	\node at (-155.1,39.3) [vertex,label=left:{$v$}] (v16) {};
	\node at (-133.9,87.5) [vertex] (v17) {};
	\node at (-98.3,126.3) [vertex] (v18) {};
	\node at (-52,151.3) [vertex] (v19) {};
	\draw (v1) to (v2);
	\draw (v2) to (v3);
	\draw (v3) to (v4);
	\draw (v4) to (v5);
	\draw (v5) to (v6);
	\draw (v6) to (v7);
	\draw (v7) to (v8);
	\draw (v8) to (v9);
	\draw (v9) to (v10);
	\draw (v10) to (v11);
	\draw (v11) to (v12);
	\draw (v12) to (v13);
	\draw (v13) to (v14);
	\draw (v14) to (v15);
	\draw (v15) to (v16);
	\draw (v16) to (v17);
	\draw (v17) to (v18);
	\draw (v18) to (v19);
	\draw (v19) to (v1);
	\draw (v2) to (v10);
	\draw (v10) to (v16);
	\draw (v2) to (v16);
\end{tikzpicture}
		\caption{The graph $\Lambda(4,5,7)$}
		\label{fig:Lambda457}
	\end{figure}
	It is immediate that if any of the parameters $a,b,c$ are odd then the graph is not doubly Eulerian, because two circuits starting from the midpoint of that division of the Hamiltonian cycle will be at two of the valency 4 vertices at the same time.
	\begin{problem}
		Prove necessary and sufficient conditions on the parameters $a,b,c$ for the graph $\Lambda(a,b,c)$ to be doubly Eulerian. In particular, prove that there exists a Hamiltonian edge-minimal doubly Eulerian graph of order $n$ and edge excess 3 for all odd $n\geq 15$.
	\end{problem}
	The next problem relates to Eulerian graphs which are not doubly Eulerian. From Table~\ref{tab:ind4reg}, 4-regular graphs with avoidance index 1 seem to be rare.
	But in Theorem~\ref{thm:4regind1} we gave a simple proof that there exists such a graph for every order $n\geq 5$. For $n\geq 11$, these graphs have a cut vertex;
	indeed this feature is a critical part of the proof. However it would be interesting to discover whether there are other graphs.
	\begin{problem}
		Investigate whether for every $n\geq 14$, there exists a 4-regular graph of order $n$ which is not doubly Eulerian and does not contain a cut vertex.	
	\end{problem}
	
	Finally, in Section~\ref{sec:index} we extended the definition of the avoidance index of an Eulerian graph to that of a vertex, and thus to the mean avoidance index of a graph. Apart from Theorem~\ref{thm:K2r2s} we did not prove any results about the mean avoidance index and we have not investigated this concept in any detail. There are undoubtedly many questions that could be asked, and the whole field is completely open. So the final problem is more general.
	\begin{problem}
		Investigate the concepts of the avoidance index of a vertex and the mean avoidance index.
	\end{problem}
	

\section*{Statements and declarations}
The work of J. Tuite was supported by EPSRC grant EP/W522338/1 and London Mathematical Society grant ECF-2021-27.

The authors have no relevant financial or other conflicts of interest to disclose.

Data files produced during this research are available from the corresponding author on request.
\end{document}